\documentclass[11pt]{amsart}
\usepackage{fullpage}
\usepackage{amsmath, amssymb, amsthm}
\usepackage[foot]{amsaddr}

\usepackage{graphicx}
\usepackage{xcolor}
\usepackage{tikz}
\usepackage[colorlinks=true,linkcolor=red,citecolor=blue,urlcolor=cyan]{hyperref}

\usepackage{algorithm}
\usepackage{algpseudocode}
\algrenewcommand\algorithmicrequire{\textbf{Input:}}
\algrenewcommand\algorithmicensure{\textbf{Output:}}

\newtheorem{theorem}{Theorem}[section]
\newtheorem{proposition}[theorem]{Proposition}
\newtheorem{lemma}[theorem]{Lemma}
\newtheorem{corollary}[theorem]{Corollary}
\theoremstyle{remark}
\newtheorem{remark}[theorem]{Remark}

\newcommand{\R}{\mathbb{R}}
\newcommand{\Hess}{\operatorname{Hess}}
\newcommand{\tr}{\operatorname{tr}}
\newcommand{\diam}{\operatorname{diam}}
\renewcommand{\v}{\mathbf{v}}
\newcommand{\x}{\mathbf{x}}
\newcommand{\w}{\mathbf{w}}
\newcommand{\y}{\mathbf{y}}
\renewcommand{\d}{\mathbf{d}}
\newcommand{\p}{\mathbf{p}}
\newcommand{\e}{\mathbf{e}}
\newcommand{\calQ}{\mathcal{Q}}
\renewcommand{\P}{\mathcal{P}}
\newcommand{\bzero}{\mathbf{0}}

\usepackage[backend=biber,giveninits=true,sorting=nyt,style=alphabetic,natbib=true,maxcitenames=2,maxbibnames=10,url=false,doi=true,backref=false]{biblatex}
\renewbibmacro{in:}{\ifentrytype{article}{}{\printtext{\bibstring{in}\intitlepunct}}}

\begin{document}
\title{Maximizing the fundamental Laplace--Neumann \\ eigenvalue on quadrilaterals}

\author{Ryoki Endo}
\address{Faculty of Science, Niigata University, Niigata, Japan. JSPS Research Fellow (PD). }
\email{endo@m.sc.niigata-u.ac.jp}
\thanks{Ryoki Endo was supported by JSPS KAKENHI Grant Number JP24KJ1170.}
\author{Braxton Osting}
\address{Department of Mathematics, University of Utah, Salt Lake City, UT, USA}
\email{osting@math.utah.edu}
\thanks{Braxton Osting acknowledges partial support from NSF DMS-2136198 and DMS-2513175.}

\date{\today}


\begin{abstract} 
We prove that among convex planar quadrilaterals of equal area, the square uniquely maximizes the first nonzero Laplace--Neumann eigenvalue. This is the quadrilateral case of the Neumann analogue of the long-standing P\'olya--Szeg\H{o} conjecture, which asserts that the regular $n$-gon minimizes the first Laplace--Dirichlet eigenvalue among $n$-gons of equal area. The proof uses a Rayleigh--Ritz lower bound, obtained from the span of the first five nonconstant Neumann modes of the square, on a smooth auxiliary functional whose minimization implies the original maximization.
Local maximality at the square follows from a symmetry-reduced Hessian, computed in closed form, together with a certified second-order difference-quotient test that establishes the local inequality on an explicit ball around the square. Global maximality is then established by a certified box covering of the parameter space, certifying on each box a direct upper bound on the first Rayleigh--Ritz eigenvalue. Both certifications operate entirely on integrals over a fixed reference square and avoid \emph{a posteriori} finite-element bounds on the perturbed quadrilateral. 
\end{abstract}

\maketitle

\section{Introduction}\label{s:Intro}

P\'olya conjectured that the regular $n$-gon minimizes the first Laplace--Dirichlet eigenvalue among $n$-gons of prescribed area~\cite{Polya1948,PolyaSzego1951}. Despite its apparent simplicity, the conjecture is notoriously difficult: only the cases $n = 3$ and $n = 4$ have been settled, both via Steiner symmetrization~\cite{Henrot2006}. Bogosel and Bucur~\cite{BogoselBucur2024} showed that for each $n \geq 5$ the conjecture reduces to finitely many certified numerical computations, and that local minimality of the regular $n$-gon reduces to a single such computation, which they carry out for $n = 5, 6, 7, 8$. 

By analogy with the Faber--Krahn and Szeg\H{o}--Weinberger inequalities, the regular $n$-gon is suspected to be a maximizer for the first nonzero Laplace--Neumann eigenvalue in the same class; for the maximization of higher Neumann eigenvalues among general planar domains, see~\cite{BucurMartinetOudet2023}. For $n = 3$ the equilateral triangle was shown to be the maximizer by Laugesen and Siudeja, via a transplantation argument relying on affine deformations especially suited to triangles~\cite{LaugesenSiudeja2009}. We consider the problem for quadrilaterals and prove the following result.

\begin{theorem} \label{t:Main}
Among convex planar quadrilaterals of equal area, the square $\square$ is, up to congruence, the unique maximizer of the first nonzero Laplace--Neumann eigenvalue $\mu_1$, with $|\square| \mu_1(\square) = \pi^2$.
\end{theorem}

The result of Theorem~\ref{t:Main} is not obvious as non-uniqueness and symmetry-breaking do occur in related shape optimization problems.  Indeed, $\mu_1$ is the $p = 2$ member of the family
$
\Lambda_p(\Omega) := \min\left\{ \frac{\int_\Omega |\nabla u|^p\, dx}{\int_\Omega |u|^p\, dx} \colon \int_\Omega |u|^{p-2}\, u\, dx = 0 \right\}
$
of first nontrivial Neumann eigenvalues of the $p$-Laplacian $-\Delta_p u = \operatorname{div}\!\bigl(|\nabla u|^{p-2}\nabla u\bigr)$; 
at $p = 2$ the constraint reads $\int_\Omega u = 0$ and $\Lambda_2(\Omega) = \mu_1(\Omega)$. The scale-invariant quantity is $|\Omega|^{p/2}\Lambda_p(\Omega)$, which at $p = 2$ is the $|\Omega|\,\mu_1(\Omega)$ maximized in Theorem~\ref{t:Main}. At both endpoints of this family, that conclusion fails.
Since $\Lambda_p^{1/p} \to \frac{2}{\diam(\Omega)}$ as $p \to \infty$~\cite{EspositoKawohlNitschTrombetti2015}, the corresponding maximization among $n$-gons of fixed area is the largest-small-$n$-gon problem: among $n$-gons of unit diameter, find the one of largest area. For hexagons, Graham showed the regular hexagon is not optimal: a symmetry-breaking hexagon has strictly larger area~\cite{Graham1975}. For quadrilaterals the square is a maximizer, but not unique; a convex quadrilateral of diameter $D$ has area at most $\tfrac{1}{2}D^2$, with equality exactly when its diagonals are perpendicular and both have length $D$.

At $p = 1$, $\Lambda_1(\Omega) = h(\Omega) := \inf_{E} \frac{P(E; \Omega)}{\min(|E|, |\Omega \setminus E|)}$, the relative Cheeger constant, with $P(E; \Omega)$ the perimeter of $E$ inside $\Omega$~\cite{KawohlFridman2003}; for the \emph{classical} Cheeger constant, by contrast, the regular $n$-gon is optimal among $n$-gons of given area~\cite{BucurFragala2016}. Here the square fails to be optimal; for the scale-invariant quantity $|Q|^{1/2} h(Q)$ the square gives exactly $2$, while a particular isosceles trapezoid has a slightly larger value \cite{Dorin2026}. 
Symmetry thus breaks at $p = 1$ and $p = \infty$, whereas at $p = 2$ the square is the unique maximizer (Theorem~\ref{t:Main}).

\subsection*{Outline of the paper}
We first discuss a parameterization of convex quadrilaterals in Section~\ref{s:QuadParam}. Our proof of Theorem~\ref{t:Main}, overviewed in Section~\ref{s:RayleighRitz}, reduces the eigenvalue maximization to a finite-dimensional problem and proceeds in two steps, one local and one global. As an auxiliary quantity we use $F(Q) := \frac{1}{2|Q|}\bigl(\frac{1}{\mu_1(Q)} + \frac{1}{\mu_2(Q)}\bigr)$, which is smooth even at the square (where $\mu_1$ and $\mu_2$ coincide) and satisfies $F(Q) \le \tfrac{1}{|Q|\mu_1(Q)}$. Hence, if the square minimizes $F$, it maximizes $|Q|\mu_1(Q)$. A Rayleigh--Ritz lower bound from a $5$-dimensional trigonometric subspace gives a real-analytic function $f_2$ on a four-dimensional parameter space of convex quadrilaterals, $\p \mapsto Q_\p$ with $Q_\bzero = \square$ (Section~\ref{s:QuadParam}), satisfying $f_2 \le F$ with equality at the square. Combined with Kr\"oger's diameter inequality, which restricts attention to a bounded set $\P_{\mathrm K}$ of parameters, the proof reduces to showing $|Q_\p|\mu_1(Q_\p) < \pi^2$ on $\P_{\mathrm K} \setminus \{\bzero\}$.

Section~\ref{app:common} develops the computational framework common to the local and global steps: the stiffness and mass matrices $K(\p), M(\p)$ of the Rayleigh--Ritz approximation are pulled back to the fixed reference square as explicit parameter-dependent integrals, yielding a real-analytic $f_2$ with certified cluster margins on the local ball (Proposition~\ref{l:M3-bound}).

In Section~\ref{s:LocalMax} we establish the square as a strict local maximizer on a closed ball of explicit radius 
$|\p| = \rho^\sharp := \lambda_{\min}(H_2) = \tfrac{3232}{27\pi^6} \approx 0.1245$ in the parameter space. The dihedral symmetry of the square forces the gradient of $f_2$ to vanish, and we compute the Hessian $H_2 := \Hess_\p f_2(\bzero)$ in closed form, by second-order perturbation theory of the degenerate eigenvalue pair, with entries in $\mathbb{Q}(\pi^2)$, and verify that it is positive-definite. We then certify $f_2(\p) > \pi^{-2}$ on the ball directly, by a second-order difference-quotient test along rays $\p = t\e$ through the square. It uses no third derivative and no enclosure of the Hessian field, and is sound at the cluster degeneracy because it works with the symmetric functions of the bottom pair.

In Section~\ref{s:GlobalMax} we promote the local statement to a global one. The bounded parameter region outside the ball of radius $|\p| = \tfrac{\rho^\sharp}{2}$ is covered by finitely many axis-aligned boxes, overlapping the local ball on the collar $\tfrac{\rho^\sharp}{2} \leq \|\p\| \leq \rho^\sharp$. On each box the smallest generalized eigenvalue $\lambda_1(\p)$ of the pencil is enclosed by the verified generalized-eigenvalue solver \texttt{veigs}~\cite{LiuYanagisawa2025veigs}, in an index-certified enclosure. The resulting certified upper bound yields $|Q_\p|\lambda_1(\p) < \pi^2$ and hence $|Q_\p|\mu_1(Q_\p) < \pi^2$. The matrix entries used in all certifications are defined by parameter-dependent integrals over the \emph{fixed} reference square; see Section~\ref{app:integral-rep}.

Supporting material is collected in the appendices: Appendix~\ref{app:diff-quotient} describes the certified computations establishing the spectral margins and the second-order difference-quotient test that certifies $f_2 > \pi^{-2}$ on the local ball (Proposition~\ref{l:M3-bound} and Corollary~\ref{t:LocalMinExplicit}), and Appendix~\ref{app:implementation} documents the implementation of the certified per-box test.

\subsection*{Numerical certificates.} The proof of Theorem~\ref{t:Main} has two kinds of numerical input: 
(1) the spectral margins and the second-order difference-quotient test certifying $f_2 > \pi^{-2}$ on the ball $\overline{B(\bzero, \rho^\sharp)}$, used in the local step (see Proposition~\ref{l:M3-bound} and Appendix~\ref{app:diff-quotient}) and
(2) a finite box cover of $\Omega_{\mathrm{II}}$, used in the global step (see Theorem~\ref{t:box-cover-terminates} and Appendix~\ref{app:implementation}). Throughout, we distinguish computations performed in floating point, from \emph{validated} computations, performed in interval arithmetic with rigorous enclosures. Both inputs are certified using MATLAB, \textsc{Intlab}~\cite{Rump1999INTLAB}, and the index-aware generalized-eigenvalue solver \texttt{veigs}~\cite{LiuYanagisawa2025veigs}; the code is available~\cite{SupplementaryCode}.

\section{Parameterization of quadrilaterals}
\label{s:QuadParam}

Let $\calQ$ denote the class of convex planar quadrilaterals. Up to translation, rotation, and dilation, a convex quadrilateral is determined by four parameters, as we now describe.

\subsection*{Parameterization}
Let $\v_1^0, \v_2^0, \v_3^0, \v_4^0$ denote the vertices of the unit square $\square = \left[-\tfrac{1}{2}, \tfrac{1}{2}\right]^2$ (centered at the origin), in counter-clockwise order:
\begin{equation*}
\v_1^0 = \begin{pmatrix} -\tfrac{1}{2} \\ -\tfrac{1}{2} \end{pmatrix}, \quad
\v_2^0 = \begin{pmatrix} +\tfrac{1}{2} \\ -\tfrac{1}{2} \end{pmatrix}, \quad
\v_3^0 = \begin{pmatrix} +\tfrac{1}{2} \\ +\tfrac{1}{2} \end{pmatrix}, \quad
\v_4^0 = \begin{pmatrix} -\tfrac{1}{2} \\ +\tfrac{1}{2} \end{pmatrix}.
\end{equation*}
For $\p = (a, b, c, d) \in \P \subset \mathbb R^4$, we denote by $Q_{\p} \in \calQ$ the quadrilateral with vertices $\v_i = \v_i^0 + \boldsymbol{\delta}_i(\p)$, where the four vertex displacements $\boldsymbol{\delta}_i(\p) \in \mathbb R^2$ are linear in $\p$:
{\small 
\begin{equation*}
\boldsymbol{\delta}_1(\p) = \tfrac{1}{2}\begin{pmatrix} a + d - b \\ -a + d + c \end{pmatrix},\;\;
\boldsymbol{\delta}_2(\p) = \tfrac{1}{2}\begin{pmatrix} -a + d + b \\ -a - d - c \end{pmatrix},\;\;
\boldsymbol{\delta}_3(\p) = \tfrac{1}{2}\begin{pmatrix} -a - d - b \\ a - d + c \end{pmatrix},\;\;
\boldsymbol{\delta}_4(\p) = \tfrac{1}{2}\begin{pmatrix} a - d + b \\ a + d - c \end{pmatrix}.
\end{equation*}}
Note $\p = \bzero$ corresponds to the unit square, $\square = Q_{\bzero}$. The displacements $\{\boldsymbol{\delta}_i(\p)\}_{i=1}^4$ satisfy
\begin{equation}
\label{e:normalization}
\sum_{i=1}^4 \boldsymbol{\delta}_i(\p) \;=\; 0, \qquad
\sum_{i=1}^4 \v_i^0 \cdot \boldsymbol{\delta}_i(\p) \;=\; 0, \qquad
\sum_{i=1}^4 \v_i^0 \times \boldsymbol{\delta}_i(\p) \;=\; 0,
\end{equation}
which respectively eliminate translations, dilations, and rotations of $Q_{\p}$. Here and throughout, $\x \times \y := x_1 y_2 - x_2 y_1$ denotes the scalar cross product of $\x, \y \in \R^2$.

The four shape parameters $\p = (a, b, c, d)$ each have a simple geometric interpretation (Figure~\ref{fig:modes}):
\begin{itemize}
\item $a \mapsto Q_{(a, 0, 0, 0)}$ is an axis-aligned rectangle centered at the origin of dimensions $(1 - a) \times (1 + a)$.
\item $b \mapsto Q_{(0, b, 0, 0)}$ is an isosceles trapezoid with vertical axis of symmetry: the bottom edge has length $1 + b$ while the top edge has length $1 - b$.
\item $c \mapsto Q_{(0, 0, c, 0)}$ is the analogous isosceles trapezoid with horizontal axis of symmetry: the right edge has length $1 + c$ and the left edge has length $1 - c$.
\item $d \mapsto Q_{(0, 0, 0, d)}$ is a rhombus with diagonals along the lines $x_1 = \pm x_2$, of lengths $\sqrt{2}(1 - d)$ and $\sqrt{2}(1 + d)$.
\end{itemize}
The area of $Q_{\p}$ can be computed via the shoelace formula to be
\begin{equation}
\label{e:Area}
|Q_{\p}| \;=\; 1 \;-\; a^2 \;-\; d^2.
\end{equation}

\begin{figure}[t!]
\centering
\begin{tikzpicture}[scale=1.45, every node/.style={font=\small}]
  \tikzset{sq/.style={gray!55, dashed}, qd/.style={thick, fill=blue!8}}
  \begin{scope}[xshift=0cm]
    \draw[sq] (-0.5,-0.5) rectangle (0.5,0.5);
    \draw[qd] (-0.35,-0.65)--(0.35,-0.65)--(0.35,0.65)--(-0.35,0.65)--cycle;
    \node at (0,-1.05) {$a$: rectangle};
  \end{scope}
  \begin{scope}[xshift=2.45cm]
    \draw[sq] (-0.5,-0.5) rectangle (0.5,0.5);
    \draw[qd] (-0.65,-0.5)--(0.65,-0.5)--(0.35,0.5)--(-0.35,0.5)--cycle;
    \draw[red!75, densely dotted, thick] (0,-0.78)--(0,0.78);
    \node at (0,-1.05) {$b$: trapezoid};
  \end{scope}
  \begin{scope}[xshift=4.90cm]
    \draw[sq] (-0.5,-0.5) rectangle (0.5,0.5);
    \draw[qd] (-0.5,-0.35)--(0.5,-0.65)--(0.5,0.65)--(-0.5,0.35)--cycle;
    \draw[red!75, densely dotted, thick] (-0.78,0)--(0.78,0);
    \node at (0,-1.05) {$c$: trapezoid};
  \end{scope}
  \begin{scope}[xshift=7.35cm]
    \draw[sq] (-0.5,-0.5) rectangle (0.5,0.5);
    \draw[qd] (-0.35,-0.35)--(0.65,-0.65)--(0.35,0.35)--(-0.65,0.65)--cycle;
    \draw[red!75, densely dotted, thick] (-0.35,-0.35)--(0.35,0.35);
    \draw[red!75, densely dotted, thick] (0.65,-0.65)--(-0.65,0.65);
    \node at (0,-1.05) {$d$: rhombus};
  \end{scope}
\end{tikzpicture}
\caption{The four shape modes. Each panel shows $Q_{\p}$ (solid) when a single
parameter is set to $0.3$ and the others to $0$, against the unit square
$\square = Q_{\bzero}$ (dashed). The dotted lines mark the symmetry axes of the
trapezoids ($b$, $c$) and the perpendicular diagonals of the rhombus ($d$).}
\label{fig:modes}
\end{figure}
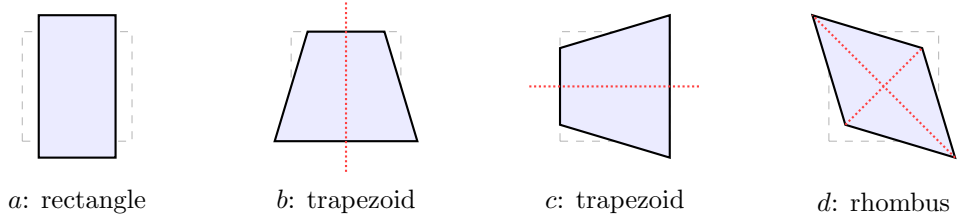

The quadrilateral $Q_{\p}$ (with our cyclic vertex ordering) is a non-degenerate convex quadrilateral if and only if 
the four cross products $c_i(\p) := (\v_{i+1} - \v_i) \times (\v_{i+2} - \v_i)$ are strictly positive; vanishing of any $c_i$ corresponds to $\v_i$ becoming collinear with its two neighbors, in which case $Q_\p$ degenerates to a triangle. A direct computation gives
\begin{equation}
\label{e:cross}
\begin{aligned}
c_1(\p) &= 1 - a^2 - d^2 \;+\; (b+c)(1-d) \;+\; a(b-c), \\
c_2(\p) &= 1 - a^2 - d^2 \;-\; (b-c)(1+d) \;-\; a(b+c), \\
c_3(\p) &= 1 - a^2 - d^2 \;-\; (b+c)(1-d) \;-\; a(b-c), \\
c_4(\p) &= 1 - a^2 - d^2 \;+\; (b-c)(1+d) \;+\; a(b+c).
\end{aligned}
\end{equation}
Note that $c_1 + c_3 = c_2 + c_4 = 2(1 - a^2 - d^2) = 2|Q_{\p}|$, recovering \eqref{e:Area}. The set $\{\p \colon c_i(\p) > 0 \text{ for all } i\}$ corresponds to convex, non-degenerate, quadrilaterals, though not uniquely: in generic position, a given similarity class of convex quadrilaterals has eight pre-images under $\p \mapsto Q_\p$, corresponding to the four choices of ``first vertex'' and two orientations. This eight-fold redundancy is the action of the dihedral group $D_4$ on the parameter space, made explicit in Section~\ref{s:RayleighRitz}. For the maximization problem, however, it is convenient to work with the closure, 
\[
\P \;:=\; \bigl\{ \p \in \mathbb R^4 \colon  c_i(\p) \geq 0 \text{ for } i = 1, 2, 3, 4 \bigr\}.
\]
Points on $\partial \P = \bigcup_i \{c_i = 0\}$ correspond to degenerate quadrilaterals where $Q_\p$ collapses to a triangle (the vertex $\v_i$ becomes collinear with its two neighbors). ($\P$ is unbounded because quadrilaterals with a very acute angle can remain convex but elongate without bound into thin shapes.)

We give a constructive proof that any convex planar quadrilateral $\widetilde Q$ (with vertices listed counter-clockwise) produces a parameter $\p \in \P$ such that $Q_{\p}$ is similar to $\widetilde Q$. 

\begin{theorem}
\label{t:surjective}
Let $\widetilde Q$ be a convex planar quadrilateral with vertices $\widetilde \v_1, \widetilde\v_2, \widetilde\v_3, \widetilde\v_4 \in \mathbb R^2$ listed in counter-clockwise order. Define:
\begin{itemize}
\item Centroid, $\bar{\widetilde\v} := \tfrac{1}{4}\sum_{i=1}^4 \widetilde\v_i$, and centered vertices, $\widetilde\w_i := \widetilde\v_i - \bar{\widetilde\v}$.
\item  Parameters 
$g := \sum_{i=1}^4 \widetilde\w_i \cdot \v_i^0$ and 
$h := \sum_{i=1}^4 \widetilde\w_i \times \v_i^0$.
\item Similarity parameters: $\lambda := 2\,(g^2 + h^2)^{-1/2} > 0$ and $\theta \in [0, 2\pi)$ defined by $(\cos\theta, \sin\theta) = \frac{(g, h)}{\sqrt{g^2 + h^2}}$.
\item Normalized vertices, $\v_i := \lambda\, R_\theta(\widetilde\w_i)$, and displacements, $\boldsymbol\delta_i := \v_i - \v_i^0$, for $i = 1, \ldots, 4$, where $R_\theta$ is counter-clockwise rotation by $\theta$.
\item Quadrilateral parameters: 
\begin{equation}\label{e:p-recovery}
\begin{aligned}
a &:= \tfrac{1}{2}\bigl(\boldsymbol\delta_{1,x} - \boldsymbol\delta_{2,x} - \boldsymbol\delta_{3,x} + \boldsymbol\delta_{4,x}\bigr), \quad &
b &:= \tfrac{1}{2}\bigl(-\boldsymbol\delta_{1,x} + \boldsymbol\delta_{2,x} - \boldsymbol\delta_{3,x} + \boldsymbol\delta_{4,x}\bigr), \\
c &:= \tfrac{1}{2}\bigl(\boldsymbol\delta_{1,y} - \boldsymbol\delta_{2,y} + \boldsymbol\delta_{3,y} - \boldsymbol\delta_{4,y}\bigr), &
d &:= \tfrac{1}{2}\bigl(\boldsymbol\delta_{1,y} - \boldsymbol\delta_{2,y} - \boldsymbol\delta_{3,y} + \boldsymbol\delta_{4,y}\bigr).
\end{aligned}
\end{equation}
\end{itemize}
Then $g^2 + h^2 > 0$ (so the construction is well-defined), $\p := (a, b, c, d) \in \operatorname{int}\P$ (so $Q_\p$ is nondegenerate), and $\widetilde\v_i = S^{-1}(\v_i(\p))$ for the similarity transform $S(\x) := \lambda R_\theta(\x - \bar{\widetilde\v})$, with inverse $S^{-1}(\x) = \lambda^{-1} R_{-\theta}(\x) + \bar{\widetilde\v}$.
\end{theorem}

\begin{proof}
We first show that $g^2 + h^2 > 0$. 
Let $\d_1 := \widetilde\w_3 - \widetilde\w_1$ and $\d_2 := \widetilde\w_4 - \widetilde\w_2$ denote the two diagonals of $\widetilde Q$. Using the vertices of the square, $\v_i^0$, we obtain 
\begin{equation*}
2g = \d_1\!\cdot(1, 1) + \d_2\!\cdot(-1, 1), \qquad 
2h = \d_1\!\cdot(1, -1) + \d_2\!\cdot(1, 1).
\end{equation*}
The linear map $(\d_1, \d_2) \mapsto (g, h)$ has $2 \times 4$ coefficient matrix $\tfrac{1}{2}\bigl(\begin{smallmatrix} 1 & 1 & -1 & 1 \\ 1 & -1 & 1 & 1 \end{smallmatrix}\bigr)$, of rank $2$, with kernel $\bigl\{(\d_1, \d_2) \in \mathbb R^4 \colon \d_1 = R_{\pi/2}\, \d_2\bigr\}$, where $R_{\pi/2}$ denotes the counter-clockwise rotation by $\tfrac{\pi}{2}$. In particular, $g = h = 0$ implies $\d_1 = R_{\pi/2}\, \d_2$, which gives $\d_1 \times \d_2 = (R_{\pi/2}\, \d_2) \times \d_2 = -|\d_2|^2 \le 0$.
However, for any counter-clockwise convex quadrilateral, the diagonals satisfy
$
\d_1 \times \d_2 \;=\; 2 \left|\widetilde Q\right| \;>\; 0,
$
since both diagonals lie inside $\widetilde Q$ and cross with positive orientation. 
Hence $g^2 + h^2 > 0$ as claimed.

Using the similarity transform, we have  $\v_i =  S(\widetilde\v_i)$. Using the identity $R_\theta(\x) \cdot \y = \x \cdot R_{-\theta}(\y)$, we compute
\begin{align*}
\sum_i \v_i \cdot \v_i^0 &= \lambda \sum_i R_\theta(\widetilde\w_i) \cdot \v_i^0 = \lambda(g\cos\theta + h\sin\theta) = \lambda\, \sqrt{g^2 + h^2} = 2, \\
\sum_i \v_i \times \v_i^0 &= \lambda(-g\sin\theta + h\cos\theta) = 0.
\end{align*}
Combined with $\sum_i \v_i = 0$ (immediate from $\sum_i \widetilde\w_i = 0$) and the identities $\sum_i \v_i^0 = 0$ and $\sum_i |\v_i^0|^2 = 2$, the displacements $\boldsymbol\delta_i = \v_i - \v_i^0$ satisfy the normalization conditions \eqref{e:normalization}.

To obtain \eqref{e:p-recovery}, consider the linear map $\Phi \colon \mathbb R^4 \to \mathbb R^8$, $\p \mapsto \left(\boldsymbol\delta_1(\p), \ldots, \boldsymbol\delta_4(\p) \right)$, given by the displacement formulas above. Substituting these formulas into the right-hand sides in \eqref{e:p-recovery} returns exactly $(a, b, c, d)$, so \eqref{e:p-recovery} defines a left inverse of $\Phi$; in particular, $\Phi$ is injective. The image of $\Phi$ lies in the subspace $W \subset \mathbb R^8$ defined by \eqref{e:normalization}; moreover, $\dim W = 8 - 4 = 4$, since the four scalar conditions are linearly independent. An injective linear map between four-dimensional spaces is an isomorphism, so $\Phi$ maps $\mathbb R^4$ isomorphically onto $W$ with inverse \eqref{e:p-recovery}. Since the displacements $\boldsymbol\delta_i$ constructed from $\widetilde Q$ satisfy \eqref{e:normalization}, they equal $\Phi(\p)$ for $\p = (a, b, c, d)$ as in \eqref{e:p-recovery}. Thus $Q_{\p} = S\left(\widetilde Q\right)$ is a convex quadrilateral with vertices in counter-clockwise order, so $c_i(\p) > 0$ for all $i$, i.e.\ $\p \in \operatorname{int}\P$.
\end{proof}

\subsection*{Reduction to a bounded set}
We  use Kr\"oger's inequality to restrict attention to a bounded subset of $\P$. For any bounded planar convex domain $\Omega$, Kr\"oger's inequality states
\[
\diam(\Omega)^2\, \mu_1(\Omega) \;\le\; 4\, j_{0,1}^2,
\]
where $j_{0,1} \approx 2.4048$ is the first positive zero of the Bessel function $J_0$~\cite{Kroger1999} (see also~\cite{HenrotMichetti2024}, where this sharp bound for planar convex domains is stated explicitly). Let $Q_{\p}$ be a quadrilateral with $|Q_{\p}|\,\mu_1(Q_{\p}) \geq \pi^2$. Dividing by $|Q_{\p}| > 0$ and then using the area bound $|Q_{\p}| \le 1$ of \eqref{e:Area},
gives
$\mu_1(Q_{\p}) \;\geq\; \frac{\pi^2}{|Q_{\p}|} \;\geq\; \pi^2 $.
Kr\"oger's inequality therefore bounds the diameter: 
\begin{equation}
\label{e:KrogerCap}
\diam(Q_{\p})^2 \;\leq\; \frac{4\, j_{0,1}^2}{\mu_1(Q_{\p})} \;\leq\; T \;:=\; \frac{4\, j_{0,1}^2}{\pi^2} \;\approx\; 2.344.
\end{equation}
In our parameterization, the squared diameter $\diam(Q_{\p})^2$ is the maximum of six explicit polynomials in $\p$, the four squared edge-lengths and the two squared diagonals,
\begin{equation}
\label{e:vertex-distances}
\begin{aligned}
|\v_1 - \v_2|^2 &= (1+b-a)^2 + (c+d)^2, \qquad &
|\v_2 - \v_3|^2 &= (1+a+c)^2 + (b+d)^2, \\
|\v_3 - \v_4|^2 &= (1-a-b)^2 + (c-d)^2, &
|\v_1 - \v_4|^2 &= (1+a-c)^2 + (b-d)^2, \\
|\v_1 - \v_3|^2 &= 2 a^2 + 2(1-d)^2, &
|\v_2 - \v_4|^2 &= 2 a^2 + 2(1+d)^2.
\end{aligned}
\end{equation}
We then define
\[
\P_{\mathrm{K}} \;:=\; \bigl\{ \p \in \P \colon |\v_i(\p) - \v_j(\p) |^2 \leq T, \; \forall
i \ne j \bigr\}.
\]
By \eqref{e:KrogerCap}, any candidate maximizer of $|Q_{\p}|\mu_1(Q_{\p})$ lies in $\P_{\mathrm{K}}$. The inequalities \eqref{e:cross} and \eqref{e:vertex-distances} make $\P_{\mathrm{K}}$ a set defined by ten quadratic inequalities. 

\begin{lemma}[Boundedness of $\P_{\mathrm{K}}$]
\label{l:PK-bounded}
Define the ellipsoidal set
\begin{equation}\label{e:ellipsoid}
P_\mathrm{E} \;:=\; \left\{ \p \in \mathbb R^4 \colon
\bigl(1 + |a|\bigr)^2 \;+\; b^2 \;+\; c^2 \;+\; d^2 \;\leq\; T \right\}
\end{equation}
and the axis-aligned box
\begin{equation}
\label{e:cube}
P_\mathrm{C} \;:=\; \left\{ \p \in \mathbb R^4 \colon  |a| \leq \alpha,\;\; |b|,\, |c| \leq \beta,\;\; |d| \leq \delta \right\},
\end{equation}
with
$\alpha = \bigl(\tfrac{T}{2} - 1\bigr)^{1/2}\approx 0.415$,
$\beta = \tfrac{1}{2}\bigl(\sqrt T - 2 + \sqrt{T + 4\sqrt T - 4}\bigr) \approx 0.822$, and
$\delta = \bigl(\tfrac{T}{2}\bigr)^{1/2} - 1  \approx 0.083$.
Then $\P_{\mathrm{K}} \subset P_{\mathrm{E}}$ and $\P_{\mathrm{K}} \subset P_{\mathrm{C}}$. In particular, $\P_{\mathrm{K}}$ is compact.
\end{lemma}

\begin{proof}
We first show $\P_{\mathrm{K}} \subset P_{\mathrm{E}}$. Adding the two pairs of opposite-edge constraints from \eqref{e:vertex-distances} gives
\begin{align*}
|\v_1 - \v_2|^2 + |\v_3 - \v_4|^2 &= 2(1-a)^2 + 2(b^2 + c^2 + d^2), \\
|\v_2 - \v_3|^2 + |\v_1 - \v_4|^2 &= 2(1+a)^2 + 2(b^2 + c^2 + d^2).
\end{align*}
For $\p \in \P_{\mathrm{K}}$, each term on the left is at most $T$, so
\[
(1-a)^2 + b^2 + c^2 + d^2 \leq T \quad\text{and}\quad (1+a)^2 + b^2 + c^2 + d^2 \leq T.
\]
Taking the maximum of the two left-hand sides yields the defining inequality of $P_{\mathrm{E}}$, since $\max((1-a)^2, (1+a)^2) = (1+|a|)^2$.

We next show $\P_{\mathrm{K}} \subset P_{\mathrm{C}}$. The bounds on $|a|$ and $|d|$ follow from the two diagonal constraints in \eqref{e:vertex-distances}, which involve only $a$ and $d$: they give $a^2 + (1 \mp d)^2 \le \tfrac{T}{2}$, hence $a^2 + (1+|d|)^2 \leq \tfrac{T}{2}$. Since $(1+|d|)^2 \ge 1$, we obtain $|a| \le \bigl(\tfrac{T}{2} - 1\bigr)^{1/2} = \alpha$; since $a^2 \ge 0$, we obtain $1 + |d| \le \bigl(\tfrac{T}{2}\bigr)^{1/2}$, i.e.\ $|d| \le \bigl(\tfrac{T}{2}\bigr)^{1/2} - 1 = \delta$.

We next bound $|b|$. Set $s := \sqrt T \approx 1.531$. We first reduce to the case $b \ge 0$: the substitution $(a, b, c, d) \mapsto (a, -b, c, -d)$ permutes the ten inequalities defining $\P_{\mathrm{K}}$ (as may be checked from \eqref{e:cross} and \eqref{e:vertex-distances}) and flips the sign of $b$, so $\sup_{\P_{\mathrm{K}}} |b| = \sup_{\P_{\mathrm{K}}} b$. Suppose then that $b \ge 0$. The edge constraint $|\v_1 - \v_2|^2 = (1 + b - a)^2 + (c+d)^2 \le T$ gives $0 < 1 + b - a \le s$, where positivity uses $|a| \le \alpha < 1$; equivalently, $1 + a \ge 2 + b - s$, and the right-hand side is positive since $s < 2 \le 2 + b$. From \eqref{e:ellipsoid}, $(1 + a)^2 + b^2 \le (1+|a|)^2 + b^2 \le T$. Combining,
\[
(2 + b - s)^2 + b^2 \;\le\; (1+a)^2 + b^2 \;\le\; T,
\]
a quadratic inequality in $b$ whose larger root is $\beta = \tfrac{1}{2}\bigl((s - 2) + \sqrt{s^2 + 4s - 4}\bigr)$; hence $b \le \beta$. The bound on $|c|$ follows from the substitution $(a, b, c, d) \mapsto (-a, -c, -b, d)$, which likewise permutes the defining inequalities of $\P_{\mathrm{K}}$ and exchanges $b$ and $c$ up to sign, so $\sup_{\P_{\mathrm{K}}} |b| = \sup_{\P_{\mathrm{K}}} |c|$. (Both substitutions are elements of the $D_4$-action on the parameter space introduced in Section~\ref{s:RayleighRitz}, induced by the reflection symmetries of the square.)

Finally, $\P_{\mathrm{K}}$ is closed, being defined by non-strict polynomial inequalities, and bounded, since $\P_{\mathrm{K}} \subset P_{\mathrm{C}}$; hence it is compact.
\end{proof}

The two supersets $P_{\mathrm{C}}$ and $P_{\mathrm{E}}$ are not nested: each contains points outside the other. The box, however, is the smaller of the two, with $\mathrm{vol}(P_{\mathrm{C}}) = 16\,\alpha \beta^2 \delta \approx 0.370$. The cross-section of $P_{\mathrm{E}}$ at fixed $a$ is a ball in $(b, c, d)$ of radius $\sqrt{T - (1+|a|)^2}$, so $\mathrm{vol}(P_{\mathrm{E}}) = \tfrac{8\pi}{3} \int_1^{\sqrt T} (T - u^2)^{3/2}\, du \approx 3.027$, roughly eight times larger. The box $P_{\mathrm{C}}$ is used in Section~\ref{s:GlobalMax} to initialize the box-covering algorithm: the certified cover of $\P_{\mathrm{K}}$ begins from a uniform subdivision of $P_{\mathrm{C}}$, boxes lying entirely outside $\P_{\mathrm{K}}$ being discarded.

\section{Rayleigh--Ritz approximation and strategy of proof}
\label{s:RayleighRitz}

In this section we set up the framework for the proof of
Theorem~\ref{t:Main}. We introduce a smooth auxiliary function $f(\p)$ whose minimization implies the original maximization and develop a
Rayleigh--Ritz lower bound $f_N(\p) \le f(\p)$ from a finite-dimensional
trigonometric subspace. We then show that the reflection symmetries of the square induce a dihedral $D_4$-action on the parameter space under which $f$ and $f_N$ are invariant; at $\p = \bzero$, this invariance forces the gradients of $f$ and $f_N$ to vanish and their Hessians to be diagonal, with only three independent entries. We conclude this section with the proof strategy.

\subsection*{Formulation of the quadrilateral maximization problem}
Using the homothety property of Laplace--Neumann eigenvalues ($\mu_k(tQ) = t^{-2} \mu_k(Q)$, $t>0$), 
our quadrilateral maximization problem can be stated as   
\begin{equation}
\label{e:maxEig1}
\max_{Q \in \mathcal Q} \,\, |Q| \mu_1(Q). 
\end{equation}
The optimal value is finite, bounded above by the Szeg\H{o}--Weinberger inequality~\cite{Szego1954,Weinberger1956}. The objective is invariant under similarity transformations, so by Theorem~\ref{t:surjective} the supremum in \eqref{e:maxEig1} equals the supremum of $|Q_\p| \mu_1(Q_\p)$ over $\p \in \P$; since the square attains the value $\pi^2$, the Kr\"oger bound \eqref{e:KrogerCap} restricts it to the compact set $\P_{\mathrm{K}}$ (Lemma~\ref{l:PK-bounded}). 

We introduce the auxiliary quadrilateral minimization problem 
\begin{equation}
\label{e:maxEig2}
\min_{Q \in \mathcal Q} \,\, F(Q), 
\qquad 
F(Q) := \frac{1}{2 |Q|} \left( \frac{1}{\mu_1(Q)} + \frac{1}{\mu_2(Q)} \right). 
\end{equation}
Note that 
$F(Q) \leq \frac{1}{|Q| \mu_1(Q)}$ 
for all $Q \in \mathcal Q$. For $\p \in \P$, we abbreviate $f(\p) := F(Q_{\p})$.
Both objectives $|Q|\mu_1(Q)$ and $F(Q)$ are scale-invariant under $Q \mapsto t Q$ ($t > 0$) by homothety; it is therefore no restriction that the parameterization $\p \mapsto Q_\p$ fixes the scale via the constraint $\sum_i \v_i^0 \cdot \boldsymbol\delta_i(\p) = 0$.
For a planar domain $\Omega \subset \mathbb R^2$ with Lipschitz boundary, 
the sum of the reciprocal Laplace--Neumann eigenvalues has the variational formulation 
\begin{equation}
\label{e:VarForm}    
\frac{1}{\mu_1(\Omega)}+ \frac{1}{\mu_2(\Omega)} 
\,\, = \, \, 
\max_{\substack{u,v \in H^1_\diamond(\Omega) \times H^1_\diamond(\Omega) \\    \int \nabla u \cdot \nabla v\, dx =0}}  E_{\Omega}(u,v), 
\qquad \qquad 
E_{\Omega} (u,v) := \frac{\int_\Omega u^2}{\int_\Omega |\nabla u|^2}+ \frac{\int_\Omega v^2}{\int_\Omega |\nabla v|^2}. 
\end{equation}
Here, $ H^1_\diamond(\Omega)
=
\left\{
u \in H^1(\Omega)
\colon
\int_\Omega u \, dx = 0
\right\}
$
denotes $H^1$ functions with zero mean. The identity \eqref{e:VarForm} is a sum-version of the Courant--Fischer characterization~\cite[Chapter III]{PolyaSchiffer1954}.
Before we use  \eqref{e:maxEig2} to study \eqref{e:maxEig1}, we  recall properties of the Laplace--Neumann eigenvalue problem on the square. 

\subsection*{Laplace--Neumann eigenvalues of the square}
The Laplace--Neumann eigenpairs of the square $\square = Q_\bzero = [-\tfrac{1}{2}, \tfrac{1}{2}]^2$ are given by
\begin{align*}
\mu_{m,n} &= \pi^2(m^2 + n^2), \\
u_{m,n}(x_1, x_2) &= \cos\bigl(m\pi(x_1 + \tfrac{1}{2})\bigr)\, \cos\bigl(n\pi(x_2 + \tfrac{1}{2})\bigr), 
\qquad m, n \in \mathbb{N}_{\ge 0}.
\end{align*}
The lowest eigenvalue is $\mu_0 = 0$, and the next two satisfy $\mu_1(\square) = \mu_2(\square) = \pi^2$, an eigenvalue of multiplicity two corresponding to the two-dimensional eigenspace spanned by $u_{1,0}$ and $u_{0,1}$. It follows that $f(\bzero) = F(Q_\bzero) = \tfrac{1}{|\square|\, \mu_1(\square)} = \tfrac{1}{\pi^2}$. Combined with the inequality $F(Q) \le \frac{1}{|Q|\, \mu_1(Q)}$ for all $Q \in \calQ$, this shows that if the square $\square$ is a minimizer for \eqref{e:maxEig2}, then it is a maximizer for \eqref{e:maxEig1}. 

\subsection*{Rayleigh--Ritz approximation}
The variational formulation in \eqref{e:VarForm} expresses $\frac{1}{\mu_1} + \frac{1}{\mu_2}$ as a maximum over admissible test pairs $(u,v) \in H^1_\diamond(Q) \times H^1_\diamond(Q)$. 
Following the classical Rayleigh--Ritz method, restricting the maximum to any chosen \emph{finite-dimensional} subspace 
$V \subset H^1_\diamond(Q)$ 
yields a computable lower bound.
Thus, for $V \subset H^1_\diamond(Q)$, we denote 
\[
F^V(Q) 
\;:=\; 
\frac{1}{2|Q|}  \, 
\max_{\substack{u,v \in V \\ \int \nabla u \cdot \nabla v\, dx = 0}}\!\! E_{Q} (u,v)
\]
and observe that $F^V(Q) \le F(Q)$ for every $Q \in \mathcal Q$.

Choose a basis $\{\varphi_i\}_{i = 1}^k$ of $V$ and construct the \emph{mass} and \emph{stiffness matrices},  
\[
(M^V)_{ij} \;=\; \int_Q \varphi_i  \varphi_j \, dx
\qquad \textrm{and} \qquad 
(K^V)_{ij} \;=\; \int_Q \nabla \varphi_i \cdot \nabla \varphi_j \, dx, \qquad i,j =1,\ldots,k.  
\]
Let $\lambda_1 \le \lambda_2 \le \cdots \le \lambda_k$ denote the eigenvalues of the finite-dimensional generalized eigenvalue problem $K^V u = \lambda M^V u$.
For each integer $N \ge 1$, define the trigonometric subspace
\[
\widetilde V^N
\;:=\;
\mathrm{span}\Bigl\{\, \cos\bigl(m\pi(x_1 + \tfrac{1}{2})\bigr) \cos\bigl(n\pi(x_2 + \tfrac{1}{2})\bigr) \colon (m,n) \in \mathbb{Z}_{\ge 0}^2,\ 0 < m+n \le N\,\Bigr\} 
\;\subset\; H^1(\mathbb{R}^2).
\]
Observe that $\dim \widetilde V^N = \tfrac{N(N+3)}{2}$, and 
$\widetilde V^1 = \mathrm{span}\{\cos(\pi(x_1+\tfrac{1}{2})),\, \cos(\pi(x_2+\tfrac{1}{2}))\}$ 
is the two-dimensional subspace spanned by the lowest nontrivial Neumann eigenfunctions of the unit square. The basis functions of $\widetilde V^N$ are defined in the ambient plane and, for each quadrilateral $Q_{\p}$, are restricted to $Q_{\p}$ to define a $\p$-dependent test subspace
\[
V^N_{\p} 
\;:=\; 
\mathrm{span}\Bigl\{ \,\varphi|_{Q_{\p}} - \tfrac{1}{|Q_{\p}|}\!\textstyle{\int_{Q_{\p}}}\!\varphi\, dx \colon  \varphi \in \widetilde V^N\,\Bigr\} 
\;\subset\; H^1_\diamond(Q_{\p}).
\]
We abbreviate $f_N(\p) := F^{V^N_{\p}}(Q_{\p})$.

From the definitions above, we have
\begin{equation}
\label{e:KeyInequality}
\begin{aligned}
& f_N(\p) \;\le\; F(Q_{\p}) \;\le\; \frac{1}{|Q_{\p}|\, \mu_1(Q_{\p})}, \qquad \qquad  \forall \p \in \P
\quad \textrm{and} \quad  \\ 
& f_N(\bzero) \;=\; 
F(Q_{\bzero}) \;=\; 
\frac{1}{|Q_{\bzero}|\, \mu_1(Q_{\bzero})} \;=\; 
\frac{1}{\pi^2}.
\end{aligned}
\end{equation}
The first inequality is the Rayleigh--Ritz bound applied to $V^N_{\p} \subset H^1_\diamond(Q_{\p})$, and the second follows from $F(Q) \le \frac{1}{|Q|\mu_1(Q)}$, $\forall Q \in \calQ$. The equality at $\p = \bzero$ holds because $\widetilde V^N$ contains the two nontrivial first Neumann eigenfunctions of the unit square, which at $\p = \bzero$ are $L^2$-orthogonal, $H^1$-orthogonal, and have zero mean, so they realize the maximum in \eqref{e:VarForm}.

A brief remark on smoothness of $f$ and $f_N$. The proof of Theorem~\ref{t:Main} works entirely with the Rayleigh--Ritz approximant $f_N$, not with $f$ itself: the chain \eqref{e:KeyInequality} shows that $f_N(\p) > \tfrac{1}{\pi^2}$ already implies $|Q_\p|\mu_1(Q_\p) < \pi^2$, with no derivative of $f$ entering the argument. Smoothness of $f_N$ follows from finite-dimensional perturbation theory. By the integral representation of Section~\ref{app:integral-rep}, the matrices $K(\p), M(\p)$ are real-analytic functions of $\p$ on a \emph{fixed} reference square. Moreover, elementary symmetric functions of a separated cluster of generalized eigenvalues of a smoothly-varying definite matrix pencil are smooth in the parameter~\cite[Ch.~VII]{Kato1995},~\cite[\S 6.3]{StewartSun1990}. It follows that $f_N$ is real-analytic wherever the gap between the eigenvalue cluster $\{\lambda_1,\lambda_2\}$ and the higher eigenvalues is preserved (see Lemma~\ref{l:f2-smooth}). The same cluster argument applies to $f$ itself: the Neumann form pulled back to the fixed square by the map $\Phi_\p$ of Section~\ref{app:pullback} is a real-analytic family of closed forms in the sense of Kato~\cite[Ch.~VII]{Kato1995}, whose bottom nontrivial cluster $\{\mu_1, \mu_2\}$ remains isolated near the square, so $f$ is also smooth near $\bzero$, a fact used only in the symmetry discussion below.

\subsection*{Symmetry properties of the parameterization \texorpdfstring{$\p \mapsto Q_\p$}{p -> Qp}}
We now explain why the parameterization $\p \mapsto Q_\p$ enjoys symmetries that simplify our calculations of both the auxiliary function $f$ and its Rayleigh--Ritz approximant $f_N$. The unit square $\square = Q_\bzero$ is invariant under the dihedral group $D_4$ generated by the two axis-aligned reflections $\sigma_h(x_1, x_2) = (-x_1, x_2)$ and $\sigma_v(x_1, x_2) = (x_1, -x_2)$ together with the diagonal reflection $\rho(x_1, x_2) = (x_2, x_1)$. Each of these symmetries maps $Q_{\p}$ to a congruent quadrilateral; relabeling its vertices counter-clockwise and recovering the parameters via \eqref{e:p-recovery} yields a new parameter point, the resulting map being a linear involution:
\begin{align*}
\sigma_h^* &\colon (a, b, c, d) \mapsto (a,\, b,\, -c,\, -d), \\
\sigma_v^* &\colon (a, b, c, d) \mapsto (a,\, -b,\, c,\, -d), \\
\rho^* &\colon (a, b, c, d) \mapsto (-a,\, -c,\, -b,\, d).
\end{align*}
Each of the three reflections reverses cyclic order, so a relabeling convention is needed; we use the same rule for all three: relabel the image vertices counter-clockwise, starting from the image vertex that occupies the bottom-left corner at $\p = \bzero$. This gives the vertex sequences
\[
\bigl(\sigma_h(\v_2), \sigma_h(\v_1), \sigma_h(\v_4), \sigma_h(\v_3)\bigr), \quad
\bigl(\sigma_v(\v_4), \sigma_v(\v_3), \sigma_v(\v_2), \sigma_v(\v_1)\bigr), \quad
\bigl(\rho(\v_1), \rho(\v_4), \rho(\v_3), \rho(\v_2)\bigr),
\]
from which \eqref{e:p-recovery} produces the three formulas above. The sign patterns can also be read off Figure~\ref{fig:modes}. With this convention, $Q_{\sigma^*\p} = \sigma(Q_\p)$ \emph{as point sets}: the relabeling matches corners at $\p = \bzero$, so orthogonality of $\sigma$ shows that the relabeled displacements of $\sigma(Q_\p)$ satisfy the normalization \eqref{e:normalization}, and \eqref{e:p-recovery} recovers them with no similarity correction. In particular $|Q_{\sigma^*\p}| = |Q_\p|$. Together the three involutions generate a group of eight signed relabelings of the parameter space (the $D_4$-action whose orbits give the eight-fold redundancy of Section~\ref{s:QuadParam}), and $\P_{\mathrm K}$ is $D_4$-invariant, since each $\sigma$ permutes the vertex distances \eqref{e:vertex-distances} of $Q_\p$.

Both $f$ and $f_N$ are invariant under each involution. For $f$ this is immediate: $f(\p) = F(Q_\p)$ depends only on the Laplace--Neumann spectrum of $Q_\p$, hence only on its congruence class. For $f_N$, writing $u = x_1 + \tfrac{1}{2}$ and $v = x_2 + \tfrac{1}{2}$, the basis functions $\psi_{m,n}(x_1, x_2) = \cos(m\pi u) \cos(n\pi v)$ satisfy
\begin{equation}\label{e:psi-rules}
\psi_{m,n} \circ \sigma_h = (-1)^m \psi_{m,n}, \qquad
\psi_{m,n} \circ \sigma_v = (-1)^n \psi_{m,n}, \qquad
\psi_{m,n} \circ \rho = \psi_{n,m},
\end{equation}
using, for example, that $\sigma_h$ sends $u \mapsto 1 - u$ and $\cos(m\pi(1-u)) = (-1)^m \cos(m\pi u)$. Each rule sends a basis function to $\pm$ another and preserves the index set $\{(m,n) \colon 0 < m + n \le N\}$, so $\widetilde V^N \circ \sigma = \widetilde V^N$. Composition with $\sigma$ is therefore an isometry of $H^1_\diamond(Q_{\sigma^*\p})$ onto $H^1_\diamond(Q_\p)$ for both the $L^2$- and the Dirichlet inner products, and it maps $V^N_{\sigma^*\p}$ onto $V^N_\p$ (the means transform compatibly, the areas being equal); hence the restricted maximum defining $F^V$ is unchanged, and $f_N(\sigma^*\p) = f_N(\p)$.

For the remainder of this subsection, let $g$ denote either $f$ or $f_N$; the arguments below apply uniformly to both, each being smooth near $\bzero$ (see the remark closing the previous subsection). The invariance identities kill the gradient at $\bzero$ by the chain rule alone: differentiating $g(a, b, c, d) = g(a, -b, c, -d)$ in $b$ at $\bzero$ gives $\partial_b g(\bzero) = -\partial_b g(\bzero) = 0$, and differentiating in $d$ gives $\partial_d g(\bzero) = 0$; the $\sigma_h^*$-identity likewise gives $\partial_c g(\bzero) = 0$; and differentiating $g(a, b, c, d) = g(-a, -c, -b, d)$ in $a$ gives $\partial_a g(\bzero) = 0$. Hence $\nabla g(\bzero) = \bzero$.

The same identities constrain the Hessian, $H = \Hess_\p\, g(\bzero)$. The involutions $\sigma_h^*$ and $\sigma_v^*$ are the diagonal sign matrices $\mathrm{diag}(1, 1, -1, -1)$ and $\mathrm{diag}(1, -1, 1, -1)$; writing $\sigma = \mathrm{diag}(\varepsilon_a, \varepsilon_b, \varepsilon_c, \varepsilon_d)$ for either one, invariance gives $\sigma H \sigma = H$, that is, $H_{xy} = \varepsilon_x \varepsilon_y H_{xy}$ entrywise. The sign patterns $(+,+), (+,-), (-,+), (-,-)$ of $a, b, c, d$ under $(\sigma_h^*, \sigma_v^*)$ are pairwise distinct, so every pair $x \ne y$ has $\varepsilon_x \varepsilon_y = -1$ for one of the two, and $H$ is \emph{diagonal} in the basis $(e_a, e_b, e_c, e_d)$. Taking $\partial_b^2$ of the $\rho^*$-identity $g(a, b, c, d) = g(-a, -c, -b, d)$ at $\bzero$ gives $H_{bb} = H_{cc}$ (the sign of $-b$ squares away), so $H = \mathrm{diag}(H_{aa}, H_{bb}, H_{bb}, H_{dd})$, with at most three independent entries. In particular, this applies to the Hessian $H_N := \Hess_\p\, f_N(\bzero)$ of the Rayleigh--Ritz approximant, whose entries are computed in closed form for $N \in \{1, 2\}$ in Section~\ref{s:LocalMax} (Corollary~\ref{c:closed-form-hessian}).

\subsection*{Proof strategy for Theorem~\ref{t:Main}}
\label{ss:strategy}

By Kr\"oger's inequality \eqref{e:KrogerCap}, any $\p$ with $|Q_\p|\mu_1(Q_\p) \ge \pi^2$ lies in $\P_{\mathrm{K}}$, so to prove Theorem~\ref{t:Main} it suffices to show
\begin{equation}\label{e:f2-target}
  |Q_\p|\,\mu_1(Q_\p) \;<\; \pi^2, \qquad \forall \p \in \P_{\mathrm{K}} \setminus \{\bzero\}.
\end{equation}
Fix $N = 2$ for the remainder of the paper, and let $\lambda_1(\p) \le \cdots \le \lambda_5(\p)$ be the generalized eigenvalues of the pencil $(K(\p), M(\p))$. 
Since $V^2_\p \subset H^1_\diamond(Q_\p)$, two certificates for \eqref{e:f2-target} are available: $f_2(\p) > \tfrac{1}{\pi^2}$ implies $|Q_\p|\mu_1(Q_\p) < \pi^2$ by \eqref{e:KeyInequality}, and $|Q_\p|\lambda_1(\p) < \pi^2$ implies the same, since the smallest Rayleigh--Ritz eigenvalue over-estimates $\mu_1$. The set $\P_{\mathrm{K}}$ is compact (Lemma~\ref{l:PK-bounded}) and $\bzero \in \operatorname{int}\P$ (each $c_i(\bzero) = 1$), and we cover $\P_{\mathrm{K}} \setminus \{\bzero\}$ by two overlapping regions, using one certificate on each:

\begin{itemize}
\item \textbf{Step (I) (local).} On the closed ball $\overline{B(\bzero, \rho^\sharp)}$ of explicit radius $\rho^\sharp := \lambda_{\min}(H_2) = \tfrac{3232}{27\pi^6} \approx 0.1245$, we certify $f_2(\p) > \tfrac{1}{\pi^2}$ for $\p \ne \bzero$ by a second-order difference-quotient test (Section~\ref{s:LocalMax}, Corollary~\ref{t:LocalMinExplicit}; Appendix~\ref{app:diff-quotient}).

\item \textbf{Step (II) (global).} On $\Omega_{\mathrm{II}} := \P_{\mathrm{K}} \setminus B\bigl(\bzero, \tfrac{\rho^\sharp}{2}\bigr)$ (the \emph{open} ball is removed, so $\Omega_{\mathrm{II}}$ is compact and overlaps the local ball on the collar $\tfrac{\rho^\sharp}{2} \le \|\p\| \le \rho^\sharp$), we certify $|Q_\p|\lambda_1(\p) < \pi^2$ by a finite cover of axis-aligned boxes (Section~\ref{s:GlobalMax}, Theorem~\ref{t:box-cover-terminates}; Appendix~\ref{app:implementation}).
\end{itemize}

The two steps use different certificates because the obstruction differs by region. Near the square the bottom pair is degenerate, $\mu_1(\square) = \mu_2(\square) = \pi^2$, so $\mu_1$ alone is only Lipschitz; the symmetric two-eigenvalue functional $f_2$ is real-analytic through the degeneracy (Lemma~\ref{l:f2-smooth}), and a second-order difference-quotient test certifies $f_2 > \pi^{-2}$ on a ball about the square (Step~(I)). Away from the square $f_2$ is no longer suitable: the larger second eigenvalue can drive it below $\tfrac{1}{\pi^2}$, and the bottom pair need not remain isolated, so the analyticity of Lemma~\ref{l:f2-smooth} is lost. Step~(II) instead bounds the single smallest eigenvalue directly, wherever $M(\p) \succ 0$, requiring neither a spectral gap nor analyticity.

Neither step computes $\mu_1$ on the perturbed quadrilateral $Q_\p$ or invokes a posteriori FEM error analysis~\cite{LiuOishi2013}: the verification operates on integrals over the fixed reference square (Section~\ref{app:integral-rep}).

\section{Properties of the Rayleigh--Ritz approximation on the reference square}\label{app:common}

In this section we assemble the computational framework used by both the local step (Section~\ref{s:LocalMax}) and the global step (Section~\ref{s:GlobalMax}): the matrices $K(\p), M(\p)$ of the Rayleigh--Ritz approximation, their pullback to the reference square and integral representation, the real-analyticity of $f_2$, and the certified cluster margins on the local ball (Proposition~\ref{l:M3-bound}).

\subsection{The pencil and its equivariance}\label{app:setup}

Recall the test subspace $V^N_\p \subset H^1_\diamond(Q_\p)$ of Section~\ref{s:RayleighRitz}, obtained by restricting the trigonometric basis $\widetilde V^N$ to $Q_\p$ and subtracting means, with $f_N(\p) = F^{V^N_\p}(Q_\p)$. Index the modes by $\mathcal I_N := \{(m, n) \in \mathbb Z_{\geq 0}^2 \colon 0 < m + n \le N\}$, fix an enumeration, and write the corresponding basis functions
\[
\psi_{m,n}(x;\p) \;:=\; \cos\bigl(m\pi(x_1 + \tfrac{1}{2})\bigr)\,\cos\bigl(n\pi(x_2 + \tfrac{1}{2})\bigr) \;-\; \frac{1}{|Q_{\p}|}\!\int_{Q_{\p}}\!\cos\bigl(m\pi(y_1 + \tfrac{1}{2})\bigr)\,\cos\bigl(n\pi(y_2 + \tfrac{1}{2})\bigr)\, dy.
\]
The mean subtraction ensures $\psi_{m,n}(\cdot;\p) \in H^1_\diamond(Q_{\p})$.

The mass and stiffness matrices $M(\p) := M^{V^N_{\p}}$ and $K(\p) := K^{V^N_{\p}}$ have entries
\[
M_{ij}(\p) \;=\; \int_{Q_{\p}} \psi_i(x;\p)\,\psi_j(x;\p)\, dx, \qquad
K_{ij}(\p) \;=\; \int_{Q_{\p}} \nabla \psi_i(x;\p) \cdot \nabla \psi_j(x;\p)\, dx.
\]
Let $\lambda_1(\p) \le \lambda_2(\p)$ be the two smallest generalized eigenvalues of the pencil $(K(\p), M(\p))$. The maximum defining $f_N(\p)$ is \eqref{e:VarForm} restricted to $V^N_{\p}$. The reversed pencil $(M(\p), K(\p))$ has eigenvalues $1/\lambda_i(\p)$, whose two largest are $1/\lambda_1(\p)$ and $1/\lambda_2(\p)$. So by the P\'olya--Schiffer argument~\cite{PolyaSchiffer1954} (see also \cite{Fan1949}), the constrained maximum
\[
\max\Bigl\{ E_{Q_{\p}}(u, v) \colon u, v \in V^N_{\p},\ \textstyle\int_{Q_{\p}} \nabla u \cdot \nabla v\, dx = 0 \Bigr\} \;=\; \frac{1}{\lambda_1(\p)} + \frac{1}{\lambda_2(\p)},
\]
attained at the eigenvectors $\xi_1, \xi_2$ of $(K(\p), M(\p))$ belonging to $\lambda_1(\p), \lambda_2(\p)$. Take $\xi_1, \xi_2$ to be $M(\p)$-orthogonal, which is possible even when $\lambda_1(\p) = \lambda_2(\p)$, as at the square. If $u, v \in V^N_{\p}$ are the associated test functions, the constraint in \eqref{e:VarForm} is exactly $K(\p)$-orthogonality of the coefficients,
\[
\int_{Q_{\p}} \nabla u \cdot \nabla v\, dx \;=\; \xi_1^\top K(\p)\,\xi_2 \;=\; \lambda_2(\p)\,\xi_1^\top M(\p)\,\xi_2 \;=\; 0 ,
\]
using $K(\p)\xi_2 = \lambda_2(\p) M(\p)\xi_2$ and then the $M(\p)$-orthogonality. The pair is thus admissible. Hence $f_N(\p) = \tfrac{1}{2|Q_{\p}|}\bigl(\tfrac{1}{\lambda_1(\p)} + \tfrac{1}{\lambda_2(\p)}\bigr)$.

The $D_4$-symmetry acts on the pencil itself, by signed permutations of the modes:
\begin{equation}\label{e:equivariance}
K(\sigma^*\p) \;=\; P_\sigma^\top\, K(\p)\, P_\sigma, \quad
M(\sigma^*\p) \;=\; P_\sigma^\top\, M(\p)\, P_\sigma, \quad
\sigma \in \{\sigma_h, \sigma_v, \rho\},\ a^2+d^2<1.
\end{equation}
Here $P_\sigma$ is the signed permutation matrix read off \eqref{e:psi-rules}: writing $\psi_i \circ \sigma = s_i\, \psi_{\pi(i)}$, it has entries $(P_\sigma)_{\pi(i),\, i} = s_i$. Indeed, for $\p \in \operatorname{int}\P$, \eqref{e:equivariance} is the change of variables $x = \sigma(y)$ over $Q_{\sigma^*\p} = \sigma(Q_\p)$, the mean-subtraction terms transforming compatibly since the areas agree; both sides are real-analytic in $\p$ (Section~\ref{app:integral-rep}), so the identity persists on the region $\{a^2+d^2<1\}$, which is connected, contains $\operatorname{int}\P$, and is invariant under the three relabelings.

\subsection{Pullback to the reference square}
\label{app:pullback}
For computational purposes we pull integrals over $Q_{\p}$ back to the reference square $\square := [-\tfrac{1}{2}, \tfrac{1}{2}]^2$ via the bilinear interpolation map $\Phi_{\p} \colon \square \to Q_{\p}$ defined by
\[
\Phi_{\p}(u, v) \;=\;
\left(\tfrac{1}{2}-u\right)\left(\tfrac{1}{2}-v\right)\,\v_1
+ \left(\tfrac{1}{2}+u\right)\left(\tfrac{1}{2}-v\right)\,\v_2 
+ \left(\tfrac{1}{2}+u\right)\left(\tfrac{1}{2}+v\right)\,\v_3
+ \left(\tfrac{1}{2}-u\right)\left(\tfrac{1}{2}+v\right)\,\v_4,
\]
which sends the corners of $\square$ to the vertices $\v_1, \v_2, \v_3, \v_4$ of $Q_{\p}$ in cyclic order. At $\p = \bzero$, $\Phi_\bzero$ is the identity on $\square = Q_\bzero$.

Writing $\Phi_{\p}(u, v) = (X(u, v; \p),\, Y(u, v; \p))$, a direct computation in our parameterization yields the formula
\[
\Phi_{\p}(u, v) - \Phi_\bzero(u, v) \;=\; \begin{pmatrix} -a\, u - d\, v - 2 b\, u v \\[2pt] -d\, u + a\, v + 2 c\, u v \end{pmatrix}.
\]
The rectangle and rhombus parameters $a, d$ enter linearly in $u, v$, while the trapezoidal parameters $b, c$ enter only through the bilinear $uv$-term. The Jacobian
\[
J(u, v; \p) \;:=\; \det D\Phi_{\p}(u, v) \;=\; 1 - a^2 - d^2 + 2(c\, u - b\, v) - 2(a b\, v + a c\, u + b d\, u - c d\, v)
\]
integrates over $\square$ to give $|Q_{\p}| = 1 - a^2 - d^2$.

The pulled-back basis is built from the products
\[
\Lambda_{m, n}(u, v; \p) \;:=\; \cos\bigl(m\pi(X(u, v; \p) + \tfrac{1}{2})\bigr)\,\cos\bigl(n\pi(Y(u, v; \p) + \tfrac{1}{2})\bigr)
\]
by subtracting the mean:
\[
\widetilde\psi_{m,n}(u, v; \p) \;:=\; \Lambda_{m, n}(u, v; \p) \;-\; \frac{1}{|Q_{\p}|}\int_\square \Lambda_{m, n}\, J\, du\, dv. 
\]
At $\p = \bzero$, this collapses to $\widetilde\psi_{m,n} = \cos(m\pi(u + \tfrac{1}{2}))\cos(n\pi(v + \tfrac{1}{2}))$, exactly the unperturbed Neumann eigenfunctions of $\square = Q_\bzero$.

\subsection{Exact integral representation of \texorpdfstring{$K(\p), M(\p)$}{K(p), M(p)}}\label{app:integral-rep}
Let $A := J\, D\Phi_\p^{-1} D\Phi_\p^{-\top}$ be the coefficient matrix of the Dirichlet form pulled back from $Q_\p$ to $\square$, where $J = \det D\Phi_\p$ is the Jacobian of Section~\ref{app:pullback} (the entries of $D\Phi_\p$ are affine in $(u, v)$ and linear in $\p$). Then, with $\nabla = (\partial_u, \partial_v)^\top$, for $(m_i, n_i), (m_j, n_j) \in \mathcal I_N$:
\begin{align*}
K_{ij}(\p) &\;=\; \int_\square \nabla \Lambda_i^\top A\, \nabla \Lambda_j\, du\, dv, \\
M_{ij}(\p) &\;=\; \int_\square \Lambda_i\, \Lambda_j\, J\, du\, dv \;-\; \frac{1}{|Q_{\p}|}\!\int_\square \Lambda_i\, J\, du\, dv\,\int_\square \Lambda_j\, J\, du\, dv.
\end{align*}
The integrand of $K_{ij}$ is, up to an overall factor $J^{-1}$, a finite sum of products of (i) polynomials in $u, v, \p$ of bidegree at most $(2, 2)$ in $(u, v)$, and (ii) trigonometric factors $\cos$ or $\sin$ evaluated at $m\pi(X + \tfrac{1}{2})$ or $n\pi(Y + \tfrac{1}{2})$, where the arguments are polynomials in $u, v, \p$ of bidegree at most $(1, 1)$ in $(u, v)$ and degree $1$ in $\p$. The integrand of $M_{ij}$ has the same form with no $J^{-1}$ factor. Since $J$ is affine in $(u, v)$ with corner values $J(\v_k^0; \p) = c_{k-1}(\p)$ (indices mod $4$), we have $J > 0$ on $\square$ exactly when $c_i(\p) > 0$ for all $i$, that is, on the interior of $\P$; the apparent $J^{-1}$ in $K_{ij}$ (the single factor carried by $A$) is nonetheless removable: since $\Lambda_i = g_i \circ \Phi_\p$ with ambient mode $g_i(x) = \cos\bigl(m_i\pi(x_1 + \tfrac{1}{2})\bigr)\cos\bigl(n_i\pi(x_2 + \tfrac{1}{2})\bigr)$, the chain rule $\nabla \Lambda_i = D\Phi_\p^\top\, (\nabla g_i) \circ \Phi_\p$ collapses the contraction to
\[
K_{ij}(\p) \;=\; \int_\square \bigl(\nabla g_i \cdot \nabla g_j\bigr)\circ\Phi_\p\; J\, du\, dv,
\]
with the Jacobian in the numerator, so $K_{ij}(\p)$ is regular up to $\partial\P$.

Although these integrals do not evaluate to elementary closed forms in $\p$ (the bilinear $uv$-term inside the cosines prevents this), the entries of $K(\p)$ are real-analytic in $\p$ on all of $\mathbb R^4$, and those of $M(\p)$ on the region $\{a^2+d^2<1\}$ where $|Q_\p|>0$ (the mean subtraction involves $1/|Q_\p|$); both
admit exact, computable Taylor expansions in $\p$ at $\bzero$ to arbitrary order; the coefficients at the square are computed in Section~\ref{app:first-order}. We work throughout with the truncation order $N = 2$, indexing modes by $\mathcal I_2 = \bigl((1, 0),\, (0, 1),\, (1, 1),\, (2, 0),\, (0, 2)\bigr)$.

\subsection{Analyticity and certified cluster margins}\label{ss:f2-smooth}

The local step (Section~\ref{s:LocalMax}) Taylor-expands $f_2$ about $\bzero$, which requires that $f_2$ be real-analytic on a ball about the square and that its derivatives there be computable. We treat regularity here, and the computation in Section~\ref{s:LocalMax}: Lemma~\ref{l:f2-smooth} gives checkable conditions for analyticity, with contour formulas for the cluster sum and product.

\begin{lemma}[Analyticity and contour representation]\label{l:f2-smooth}
Let $U \subseteq \mathbb R^4$ be an open set on which the following three conditions hold:
\begin{enumerate}
\item[(i)] $c_i(\p) > 0$ for $i = 1, \ldots, 4$;
\item[(ii)] $M(\p) \succ 0$;
\item[(iii)] the two smallest generalized eigenvalues $\lambda_1(\p) \le \lambda_2(\p)$ of the pencil $(K(\p), M(\p))$ lie in the open interval $\left(\tfrac{\pi^2}{2}, \tfrac{3\pi^2}{2}\right)$, and the remaining three exceed $\tfrac{3\pi^2}{2}$.
\end{enumerate}
Then the cluster $\{\lambda_1(\p), \lambda_2(\p)\}$ is isolated from the rest of the spectrum, and its elementary symmetric functions $\Sigma(\p) := \lambda_1(\p) + \lambda_2(\p)$ and $\Pi(\p) := \lambda_1(\p)\,\lambda_2(\p)$ are real-analytic on $U$, given by the contour integrals over $\Gamma := \{w \in \mathbb C \colon |w - \pi^2| = \tfrac{\pi^2}{2}\}$,
\begin{align*}
  \Sigma(\p) &\;=\; -\tr\!\biggl(\frac{1}{2\pi i}\!\oint_{\Gamma}\!
                       w\,\bigl(K(\p) - w M(\p)\bigr)^{-1} M(\p)\,dw\biggr), \\
  \Pi(\p) &\;=\; \frac{\Sigma(\p)^2}{2} + \frac{1}{2}\,\tr\!\biggl(\frac{1}{2\pi i}\!\oint_{\Gamma}\!
                       w^2\,\bigl(K(\p) - w M(\p)\bigr)^{-1} M(\p)\,dw\biggr).
\end{align*}
Consequently
\begin{equation}\label{e:f2-via-Sigma-Pi}
  f_2(\p) \;=\; \frac{1}{2\,|Q_\p|}\!\left(\frac{1}{\lambda_1(\p)}
                                          + \frac{1}{\lambda_2(\p)}\right)
            \;=\; \frac{\Sigma(\p)}{2\,|Q_\p|\,\Pi(\p)}
\end{equation}
is real-analytic on $U$, with $|Q_\p| = 1 - a^2 - d^2 > 0$ and $\Pi(\p) > \tfrac{\pi^4}{4} > 0$.
\end{lemma}

\begin{proof}
By (i), the entries of $K(\p), M(\p)$ are real-analytic on $U$
(Section~\ref{app:integral-rep}), and $|Q_\p| > 0$ since
$c_1 + c_3 = 2|Q_\p|$ by \eqref{e:cross}. By (ii) the pencil is
symmetric-definite, with real eigenvalues and an $M(\p)$-orthonormal
eigenbasis $\{v_j\}$; by (iii) the cluster $\{\lambda_1, \lambda_2\}$ lies
inside $\Gamma$ and the remaining three eigenvalues lie outside, so
$(K(\p) - wM(\p))^{-1}$ is real-analytic in $(w, \p)$ near $\Gamma \times U$
and both contour integrals are real-analytic on $U$~\cite[Ch.~II]{Kato1995}.
The expansion $(K - wM)^{-1}M = \sum_j (\lambda_j - w)^{-1} v_j v_j^\top M$
and the residue theorem give the displayed formulas, and
$\Pi > \tfrac{\pi^4}{4} > 0$ by (iii), so \eqref{e:f2-via-Sigma-Pi} exhibits
$f_2$ as a ratio of real-analytic functions with nonvanishing denominator.
\end{proof}

The following proposition records the certified cluster margins on the local
ball, together with their consequence for the radius of analyticity.

\begin{proposition}[Certified cluster margins and analyticity]\label{l:M3-bound}
\hspace{1em}
\begin{enumerate}
\item[(a)] On the \emph{closed} ball
$\overline{B(\bzero,\rho^\sharp)}$, we have $M(\p)\succ0$, and the
bottom cluster $\{\lambda_1,\lambda_2\}$ is isolated, with
\[
\frac{\pi^2}{2}<7.630\le\lambda_1(\p)\le\lambda_2(\p)
\le13.268<\frac{3\pi^2}{2},
\qquad
\lambda_2(\p)<\lambda_3(\p).
\]
\[
\lambda_{\min}(M(\p))\ge0.0521,\qquad
\lambda_3(\p)-\frac{3\pi^2}{2}\ge3.030.
\]
\item[(b)] Consequently, there exists $\rho_\circ > \rho^\sharp$ such that
$B(\bzero,\rho_\circ)\subset\P$, $M(\p)\succ0$, and
$\lambda_2(\p)<\lambda_3(\p)$ throughout $B(\bzero,\rho_\circ)$;
in particular, $f_2$ is real-analytic on $B(\bzero,\rho_\circ)$.
\end{enumerate}
\end{proposition}

\begin{proof}
Part~(a) is certified in Appendix~\ref{app:diff-quotient}:
Table~\ref{tab:local-second-order-certificate} gives the stated bounds
for $\lambda_1,\lambda_2$, and
Corollary~\ref{cor:local-cluster-gap} gives
$M(\p)\succ0$ and $\lambda_2(\p)<\lambda_3(\p)$.
For part~(b), the explicit estimate
\[
  c_i(\p) \;\ge\; 1 - \|\p\|^2 - \sqrt 2\, \|\p\|\bigl(1 + \|\p\|\bigr)
  - \sqrt 2\, \|\p\|^2 \;>\; 0.69,
\]
holds for $\|\p\|\le0.15$ and hence on the closed local ball.
Together with part~(a), the hypotheses (i)--(iii) of Lemma~\ref{l:f2-smooth}
hold with strict margins on the closed ball, so by compactness they persist on
$B(\bzero,\rho_\circ)$ for some $\rho_\circ>\rho^\sharp$; Lemma~\ref{l:f2-smooth}
then gives the analyticity of $f_2$. 
\end{proof}

\section{The square is a local maximizer}
\label{s:LocalMax}

We establish Step~(I) of the strategy of Section~\ref{s:RayleighRitz} in three stages. We first record the first- and second-order Taylor coefficients of the pencil $(K(\p), M(\p))$ at the square (Section~\ref{app:first-order}) and compute the Hessian $H_N = \Hess_\p\, f_N(\bzero)$ in closed form, by a second-order expansion of the cluster sum $\lambda_1^{-1} + \lambda_2^{-1}$ along rays through the square (Section~\ref{ss:cluster-sum}). We then deduce a qualitative local-maximality statement (Theorem~\ref{t:LocalMin}), together with a quantitative refinement (Corollary~\ref{t:LocalMinExplicit}) which gives the radius $\rho^\sharp$ of Step~(I) explicitly, in closed form.

\subsection{The first- and second-order matrices}\label{app:first-order}
We expand the pencil of Section~\ref{app:integral-rep} at the square,
\begin{equation}\label{e:KM-expansion}
K(\p) \;=\; K^{(0)} + \sum_x p_x\, K^{(1)}_x + \tfrac{1}{2}\sum_{x, y} p_x p_y\, K^{(2)}_{xy} + O(\|\p\|^3),
\end{equation}
and similarly for $M(\p)$, with sums over $x, y \in \{a, b, c, d\}$ (identifying $p_a = a$, $p_b = b$, $p_c = c$, $p_d = d$) and coefficients $K^{(1)}_x := \partial_{p_x} K(\bzero)$, $K^{(2)}_{xy} := \partial_{p_x}\partial_{p_y} K(\bzero)$ (likewise for $M$).

At $O(1)$, the $L^2$- and $H^1$-orthogonality of Neumann eigenfunctions on the unit square ($\p = \bzero$) gives
\[
K^{(0)} \;=\; \mathrm{diag}\!\left(\tfrac{\pi^2}{2},\; \tfrac{\pi^2}{2},\; \tfrac{\pi^2}{2},\; 2\pi^2,\; 2\pi^2\right)
\quad \textrm{and} \quad
M^{(0)} \;=\; \mathrm{diag}\!\left(\tfrac{1}{2},\; \tfrac{1}{2},\; \tfrac{1}{4},\; \tfrac{1}{2},\; \tfrac{1}{2}\right).
\]
Note that the diagonal entry of $M^{(0)}$ corresponding to mode $(1, 1)$ is $\tfrac{1}{4}$ (and not $\tfrac{1}{2}$, as for the other diagonal entries) because $\int_\square \cos^2(\pi(x_1+\tfrac{1}{2}))\,\cos^2(\pi(x_2+\tfrac{1}{2}))\,dx_1\,dx_2 = \tfrac{1}{2}\cdot\tfrac{1}{2} = \tfrac{1}{4}$.
The generalized eigenvalues $\tfrac{K^{(0)}_{jj}}{M^{(0)}_{jj}} = \pi^2(m_j^2 + n_j^2)$ are the squared Neumann frequencies of the unit square: $\pi^2, \pi^2, 2\pi^2, 4\pi^2, 4\pi^2$. The two smallest are degenerate at $\pi^2$ and span the eigenspace $V_0 := \mathrm{span}\bigl(e_{(1,0)},\, e_{(0,1)}\bigr)$. The orthogonal complement $V_0^\perp = \mathrm{span}(e_{(1,1)}, e_{(2,0)}, e_{(0,2)})$ corresponds to the higher eigenvalues $2\pi^2, 4\pi^2, 4\pi^2$.

Each first-order matrix $K^{(1)}_x, M^{(1)}_x$ ($x \in \{a, b, c, d\}$) is sparse. Table~\ref{tab:first-order} lists all nonzero entries by $(\text{row mode}, \text{column mode})$ index; any entry not listed is zero, and the symmetric mode-index pair $(j, i)$ has the same value as $(i, j)$. Direction $a$ is diagonal; directions $b, c, d$ have only off-diagonal entries. Direction $c$ is obtained from $b$ by the diagonal-reflection symmetry $\rho^*$: since $\rho^*$ sends $(a, b, c, d) \mapsto (-a, -c, -b, d)$ and acts on mode indices as $(m, n) \leftrightarrow (n, m)$, the matrices in the $c$-direction equal those in the $b$-direction transposed in the mode-index pairing and multiplied by $-1$; explicitly, $X^{(1)}_c = -\,P_\rho\, X^{(1)}_b\, P_\rho$ for $X \in \{K, M\}$, where $P_\rho$ is the permutation matrix exchanging the modes $(1,0) \leftrightarrow (0,1)$ and $(2,0) \leftrightarrow (0,2)$ (differentiate \eqref{e:equivariance} along $e_b$ and use $\rho^* e_b = -e_c$).

\begin{table}[p!]
\centering
\renewcommand{\arraystretch}{1.15}
\setlength{\tabcolsep}{6pt}
\begin{tabular}{c c l l}
\hline
$x$ & Matrix & Nonzero entries (mode-pair indices) & Value \\
\hline
$a$ & $K^{(1)}_a$ & $\mathrm{diag}$ on $(1,0), (0,1), (1,1), (2,0), (0,2)$ &
  $\bigl(\tfrac{\pi^2}{2},\, -\tfrac{\pi^2}{2},\, 0,\, 2\pi^2,\, -2\pi^2\bigr)$ \\
$a$ & $M^{(1)}_a$ & $\mathrm{diag}$ on $(1,0), (0,1), (1,1), (2,0), (0,2)$ &
  $\bigl(-\tfrac{1}{2},\, \tfrac{1}{2},\, 0,\, -\tfrac{1}{2},\, \tfrac{1}{2}\bigr)$ \\
\hline
$b$ & $K^{(1)}_b$ & $\bigl((0,1), (0,2)\bigr)$ & $\tfrac{32}{9}$ \\
$b$ & $M^{(1)}_b$ & $\bigl((1,0), (1,1)\bigr)$, $\bigl((0,1), (2,0)\bigr)$ & $\tfrac{4}{\pi^2}$ \\
    &             & $\bigl((0,1), (0,2)\bigr)$ & $\tfrac{20}{9\pi^2}$ \\
\hline
$c$ & $K^{(1)}_c$ & $\bigl((1,0), (2,0)\bigr)$ & $-\tfrac{32}{9}$ \\
$c$ & $M^{(1)}_c$ & $\bigl((0,1), (1,1)\bigr)$, $\bigl((1,0), (0,2)\bigr)$ & $-\tfrac{4}{\pi^2}$ \\
    &             & $\bigl((1,0), (2,0)\bigr)$ & $-\tfrac{20}{9\pi^2}$ \\
\hline
$d$ & $K^{(1)}_d$ & $\bigl((1,1), (2,0)\bigr)$, $\bigl((1,1), (0,2)\bigr)$ & $-\tfrac{32}{9}$ \\
$d$ & $M^{(1)}_d$ & $\bigl((1,0), (0,1)\bigr)$ & $-\tfrac{8}{\pi^2}$ \\
    &             & $\bigl((1,1), (2,0)\bigr)$, $\bigl((1,1), (0,2)\bigr)$ & $-\tfrac{56}{9\pi^2}$ \\
\hline
\end{tabular}
\caption{First-order matrices $K^{(1)}_x, M^{(1)}_x$ at $\p = \bzero$.}
\label{tab:first-order}
\end{table}

At second order we need only the pure directions $x = y$: by the $D_4$-symmetry of Section~\ref{s:RayleighRitz} the Hessian $H_N$ is diagonal in $(e_a, e_b, e_c, e_d)$, and Corollary~\ref{c:closed-form-hessian} below evaluates its entries along the coordinate directions. For each $x \in \{a, b, c, d\}$ we record $K^{(2)}_{xx} := \partial_x^2 K(\bzero)$ and $M^{(2)}_{xx} := \partial_x^2 M(\bzero)$, $5 \times 5$ symmetric matrices indexed by the modes in $\mathcal I_2$, in Table~\ref{tab:second-order-pure}.

The mean subtraction in the definition of $M(\p)$ contributes to each second-order mass matrix through the rank-one term $-2\, \mathbf{g}^x (\mathbf{g}^x)^T$, where $\mathbf{g}^x \in \mathbb R^5$ has entries $\mathbf{g}^x_i := \partial_{p_x}\! \int_\square \Lambda_i\, J\, du\, dv$ evaluated at $\p = \bzero$; explicitly,
\[
\mathbf{g}^a = (0, 0, 0, -1, 1), \quad
\mathbf{g}^b = \left(0, \tfrac{4}{\pi^2}, 0, 0, 0\right), \quad
\mathbf{g}^c = \left(-\tfrac{4}{\pi^2}, 0, 0, 0, 0\right), \quad
\mathbf{g}^d = \left(0, 0, -\tfrac{8}{\pi^2}, 0, 0\right).
\]
Table~\ref{tab:second-order-pure} lists the total second-order coefficients, mean-subtraction contributions included. The stiffness matrix vanishes in directions $a, b, c$, and is diagonal in direction $d$; in directions $b, c$, the only nonzero second derivative is the single diagonal mass entry produced by the mean subtraction.

\begin{table}[t!]
\centering
\renewcommand{\arraystretch}{1.15}
\setlength{\tabcolsep}{6pt}
\begin{tabular}{c c l l}
\hline
$x$ & Matrix & Nonzero entries (mode-pair indices) & Value \\
\hline
$a$ & $M^{(2)}_{aa}$ & $\mathrm{diag}$ on $(1,0), (0,1), (1,1)$ & $-2$ \\
    &                & $\mathrm{diag}$ on $(2,0), (0,2)$ & $-4$ \\
\hline
$b$ & $M^{(2)}_{bb}$ & $\mathrm{diag}$ on $(0,1)$ & $-\tfrac{32}{\pi^4}$ \\
\hline
$c$ & $M^{(2)}_{cc}$ & $\mathrm{diag}$ on $(1,0)$ & $-\tfrac{32}{\pi^4}$ \\
\hline
$d$ & $K^{(2)}_{dd}$ & $\mathrm{diag}$ on $(1,0), (0,1)$ & $-\pi^2$ \\
    &                & $\mathrm{diag}$ on $(1,1)$ & $-2\pi^2$ \\
    &                & $\mathrm{diag}$ on $(2,0), (0,2)$ & $-4\pi^2$ \\
$d$ & $M^{(2)}_{dd}$ & $\mathrm{diag}$ on $(1,0), (0,1), (2,0), (0,2)$ & $-1$ \\
    &                & $\mathrm{diag}$ on $(1,1)$ & $-\tfrac{128}{\pi^4}$ \\
    &                & $\bigl((2,0), (0,2)\bigr)$ & $2$ \\
\hline
\end{tabular}
\caption{Nonzero entries of the pure-direction second-order matrices $K^{(2)}_{xx}, M^{(2)}_{xx}$ at $\p = \bzero$. These matrices are diagonal except for the single off-diagonal pair $\bigl((2,0), (0,2)\bigr)$ in $M^{(2)}_{dd}$; ``$\mathrm{diag}$ on $(m,n)$'' denotes the diagonal entry at mode $(m,n)$. Any entry not listed is zero (in particular $K^{(2)}_{aa} = K^{(2)}_{bb} = K^{(2)}_{cc} = 0$), and the symmetric mode-index pair $(j, i)$ has the same value as $(i, j)$.}
\label{tab:second-order-pure}
\end{table}

\subsection{The Hessian via a cluster-sum expansion}\label{ss:cluster-sum}
We compute $H_N$ by second-order perturbation theory of the degenerate pair $\lambda_1(\bzero) = \lambda_2(\bzero) = \pi^2$, in the form of an expansion of the cluster sum $\lambda_1^{-1} + \lambda_2^{-1}$ along rays through the square. Since only symmetric functions of the pair enter, no eigenvector basis or nondegeneracy assumption is needed, and the answer is assembled from traces of the coefficient matrices of Section~\ref{app:first-order}.

The tool is the \emph{reduced resolvent}: the contour formulas of Lemma~\ref{l:f2-smooth} are exact, but evaluating them means inverting the full resolvent $(K(\p) - wM(\p))^{-1}$ along $\Gamma$, which is poorly conditioned near the cluster. The reduced resolvent inverts instead only the well-separated complementary block of the pencil, and is the form behind both cluster computations of the local step: the closed-form Hessian below and the scaled Schur complement of the difference-quotient test (Appendix~\ref{app:diff-quotient}).

\begin{lemma}[Reduced-resolvent representation of $\Sigma, \Pi$]\label{l:reduced-resolvent}
Assume the hypotheses of Lemma~\ref{l:f2-smooth} and fix a point $\p_\circ \in U$. Let $V = V(\p_\circ) \in \mathbb R^{5 \times 2}$ and $W = W(\p_\circ) \in \mathbb R^{5 \times 3}$ be $M(\p_\circ)$-orthonormal frames for the cluster eigenspace and its complement at $\p_\circ$, held \emph{fixed} as $\p$ varies. For $\p$ near $\p_\circ$ and $\lambda \in \left(\tfrac{\pi^2}{2}, \tfrac{3\pi^2}{2}\right)$, the interval enclosed by $\Gamma$, write $B(\lambda; \p) := K(\p) - \lambda M(\p)$. Then the \emph{reduced resolvent} $S_\perp(\lambda; \p) := \bigl(W^\top B(\lambda; \p)\, W\bigr)^{-1}$ is well-defined, and $\lambda_1(\p), \lambda_2(\p)$ are the two solutions there of $\det B^{\mathrm{eff}}(\lambda; \p) = 0$, where
\[
  B^{\mathrm{eff}}(\lambda; \p) \;=\; V^\top\! B(\lambda;\p)\, V \;-\; V^\top\! B(\lambda;\p)\, W \; S_\perp(\lambda; \p)\; W^\top\! B(\lambda;\p)\, V
\]
is the Schur complement of $B(\lambda; \p)$ onto the cluster frame. Thus $\Sigma$ and $\Pi$ are the sum and product of these two roots.
\end{lemma}

\begin{proof}
This is the classical Schur complement of a (degenerate) eigenvalue problem~\cite{Kato1995},~\cite[\S 6.3]{StewartSun1990}. As $[\,V \mid W\,]$ is invertible, $\lambda$ is a generalized eigenvalue of the pencil $(K(\p), M(\p))$ iff $[\,V \mid W\,]^\top B(\lambda; \p)\,[\,V \mid W\,]$ is singular. At $\p_\circ$, the lower-right block $W^\top B(\lambda; \p_\circ)\, W$ is
$\mathrm{diag}(\lambda_3 - \lambda, \lambda_4 - \lambda, \lambda_5 - \lambda)$ in the
eigenbasis there, which is invertible for $\lambda \in \left(\tfrac{\pi^2}{2}, \tfrac{3\pi^2}{2}\right)$
because $\lambda_m > \tfrac{3\pi^2}{2}$ for $m \ge 3$. By continuity, invertibility
persists for $\p$ near $\p_\circ$. The identity
$\det\!\bigl([\,V \mid W\,]^\top B\,[\,V \mid W\,]\bigr) = \det(W^\top B\, W)\,\det B^{\mathrm{eff}}(\lambda; \p)$
then makes singularity equivalent to $\det B^{\mathrm{eff}}(\lambda; \p) = 0$, and the two such
$\lambda \in \left(\tfrac{\pi^2}{2}, \tfrac{3\pi^2}{2}\right)$ are $\lambda_1(\p), \lambda_2(\p)$.
\end{proof}

Fix $\e \in S^3 := \{\e \in \mathbb R^4 \colon \|\e\| = 1\}$ and write $\p = t\e$. It is convenient to normalize the pencil by $E := M(\bzero)^{-1/2} = \mathrm{diag}\bigl(\sqrt2, \sqrt2, 2, \sqrt2, \sqrt2\bigr)$:
\[
\widetilde K(t) := E\, K(t\e)\, E, \qquad \widetilde M(t) := E\, M(t\e)\, E,
\]
so that $\widetilde M(0) = I_5$ and $\widetilde K(0) = \mathrm{diag}\bigl(\pi^2, \pi^2, 2\pi^2, 4\pi^2, 4\pi^2\bigr)$, while the generalized eigenvalues $\lambda_1(t\e) \le \cdots \le \lambda_5(t\e)$ are unchanged. Expand
\[
\widetilde K(t) = \widetilde K_0 + t\, \widetilde K_1 + t^2\, \widetilde K_2 + O(t^3), \qquad
\widetilde M(t) = I_5 + t\, \widetilde M_1 + t^2\, \widetilde M_2 + O(t^3),
\]
so that, in terms of the coefficients of \eqref{e:KM-expansion}, $\widetilde K_1 = E \bigl(\sum_x e_x\, K^{(1)}_x\bigr) E$, and along a coordinate direction $\e = e_x$ the second-order coefficient is $\widetilde K_2 = \tfrac{1}{2}\, E K^{(2)}_{xx} E$ (likewise for $M$). Set
\[
\widetilde B_1 := \widetilde K_1 - \pi^2\, \widetilde M_1,
\]
write $X|_{V_0}$, $X|_{V_0, V_0^\perp}$, $X|_{V_0^\perp}$ for the blocks of a $5 \times 5$ matrix $X$ with respect to the splitting $V_0 \oplus V_0^\perp$ of Section~\ref{app:integral-rep}, and let $\kappa_{(1,1)} = 2\pi^2$ and $\kappa_{(2,0)} = \kappa_{(0,2)} = 4\pi^2$ denote the unperturbed eigenvalues on $V_0^\perp$, so that
\[
D_\perp := \bigl(\widetilde K_0 - \pi^2 I_5\bigr)\big|_{V_0^\perp} = \mathrm{diag}\bigl(\kappa_m - \pi^2\bigr)_{m \in V_0^\perp} = \mathrm{diag}\bigl(\pi^2, 3\pi^2, 3\pi^2\bigr) \;\succ\; 0.
\]
(Thus $E = [\, E_0 \mid E_\perp \,]$ for the $M(\bzero)$-orthonormal frames $E_0, E_\perp$ of Appendix~\ref{app:diff-quotient}.)

\begin{lemma}[Cluster-sum expansion at the square]\label{l:cluster-sum}
For every $\e \in S^3$ we have $\tr\bigl(\widetilde B_1|_{V_0}\bigr) = 0$ and, as $t \to 0$,
\begin{equation}\label{e:cluster-sum}
\frac{1}{\lambda_1(t\e)} + \frac{1}{\lambda_2(t\e)}
\;=\; \frac{2}{\pi^2} \;+\; \bigl(\Pi_0(\e) + \Sigma_\perp(\e)\bigr)\, t^2 \;+\; O(t^3),
\end{equation}
where
\begin{align*}
\pi^6\, \Pi_0(\e) &:= \tr\bigl((\widetilde K_1|_{V_0})^2\bigr) \;-\; \pi^2\, \tr\bigl(\widetilde K_1|_{V_0}\, \widetilde M_1|_{V_0}\bigr) \;-\; \pi^2\, \tr\bigl(\widetilde K_2|_{V_0}\bigr) \;+\; \pi^4\, \tr\bigl(\widetilde M_2|_{V_0}\bigr), \\
\pi^6\, \Sigma_\perp(\e) &:= \sum_{k \in V_0} \sum_{m \in V_0^\perp} \frac{\pi^2}{\kappa_m - \pi^2}\, \bigl(\widetilde B_1\bigr)_{k, m}^2 \;\;\ge\;\; 0,
\end{align*}
and the remainder in \eqref{e:cluster-sum} is uniform for $\e \in S^3$.
\end{lemma}

\begin{proof}
Write $\lambda_i(t\e) = \pi^2 + \nu_i(t)$, $i = 1, 2$. Since $\widetilde M(t) \succ 0$ for small $t$ and the entries of the pencil are Lipschitz in $t$, the min--max characterization gives $\nu_i(t) = O(t)$, uniformly in $\e$, and the cluster $\{\lambda_1, \lambda_2\}$ remains isolated from the three eigenvalues near $2\pi^2, 4\pi^2, 4\pi^2$.

Let $B(t, \nu) := \widetilde K(t) - (\pi^2 + \nu)\, \widetilde M(t)$. Since $B(0, 0)|_{V_0} = 0$, $B(0, 0)|_{V_0, V_0^\perp} = 0$, and $B(0, 0)|_{V_0^\perp} = D_\perp$ is invertible, Lemma~\ref{l:reduced-resolvent} (applied at $\p_\circ = \bzero$, with the frames $E_0, E_\perp$ above) shows that, for $(t, \nu)$ near $(0, 0)$, $\pi^2 + \nu$ is a cluster eigenvalue if and only if the $2 \times 2$ Schur complement
\[
B^{\mathrm{eff}}(t, \nu) \;:=\; B|_{V_0} \;-\; B|_{V_0, V_0^\perp}\, \bigl(B|_{V_0^\perp}\bigr)^{-1} B|_{V_0^\perp\!, V_0}
\]
is singular. Substituting the expansions of $\widetilde K, \widetilde M$, on the region $\nu = O(t)$ we obtain
\[
\begin{aligned}
B^{\mathrm{eff}}(t, \nu) &\;=\; t\, \widetilde B_1|_{V_0} + t^2\, G_2 - \nu\,\bigl(I_2 + t\, \widetilde M_1|_{V_0}\bigr) + O(t^3), \\
G_2 &\;:=\; \bigl(\widetilde K_2 - \pi^2 \widetilde M_2\bigr)\big|_{V_0} - \widetilde B_1|_{V_0, V_0^\perp}\, D_\perp^{-1}\, \widetilde B_1|_{V_0^\perp\!, V_0}.
\end{aligned}
\]
Hence $\nu_1(t), \nu_2(t)$ are, up to $O(t^3)$, the two generalized eigenvalues of the symmetric-definite $2 \times 2$ pencil $\bigl(t\, \widetilde B_1|_{V_0} + t^2 G_2,\; I_2 + t\, \widetilde M_1|_{V_0}\bigr)$; conjugating by $(I_2 + t\, \widetilde M_1|_{V_0})^{-1/2} = I_2 - \tfrac{t}{2}\, \widetilde M_1|_{V_0} + O(t^2)$, these are the eigenvalues of the symmetric matrix
\[
t\, \widetilde B_1|_{V_0} \;+\; t^2 \Bigl( G_2 - \tfrac{1}{2}\bigl(\widetilde M_1 \widetilde B_1 + \widetilde B_1 \widetilde M_1\bigr)\big|_{V_0} \Bigr) \;+\; O(t^3),
\]
whence
\[
\nu_1 + \nu_2 = t\, \tr\bigl(\widetilde B_1|_{V_0}\bigr) + t^2 \Bigl( \tr G_2 - \tr\bigl(\widetilde M_1|_{V_0}\, \widetilde B_1|_{V_0}\bigr) \Bigr) + O(t^3),
\qquad
\nu_1^2 + \nu_2^2 = t^2\, \tr\bigl((\widetilde B_1|_{V_0})^2\bigr) + O(t^3).
\]
By inspection of Table~\ref{tab:first-order}, each $K^{(1)}_x$ and $M^{(1)}_x$ has zero trace on $V_0$, so $\tr\bigl(\widetilde B_1|_{V_0}\bigr) = 0$ for every $\e$; in coordinates, this recovers the $D_4$-symmetry statement $\nabla f_N(\bzero) = \bzero$ of Section~\ref{s:RayleighRitz}. Therefore
\[
\frac{1}{\lambda_1} + \frac{1}{\lambda_2}
= \sum_{i = 1, 2} \Bigl( \frac{1}{\pi^2} - \frac{\nu_i}{\pi^4} + \frac{\nu_i^2}{\pi^6} \Bigr) + O(t^3)
= \frac{2}{\pi^2} + t^2 \biggl( \frac{\tr\bigl((\widetilde B_1|_{V_0})^2\bigr)}{\pi^6} - \frac{\tr G_2 - \tr\bigl(\widetilde M_1|_{V_0}\, \widetilde B_1|_{V_0}\bigr)}{\pi^4} \biggr) + O(t^3).
\]
Expanding $\widetilde B_1 = \widetilde K_1 - \pi^2 \widetilde M_1$, the terms in $\tr\bigl((\widetilde M_1|_{V_0})^2\bigr)$ cancel, leaving
\[
\frac{\tr\bigl((\widetilde B_1|_{V_0})^2\bigr)}{\pi^6} + \frac{\tr\bigl(\widetilde M_1|_{V_0}\, \widetilde B_1|_{V_0}\bigr)}{\pi^4}
= \frac{\tr\bigl((\widetilde K_1|_{V_0})^2\bigr) - \pi^2\, \tr\bigl(\widetilde K_1|_{V_0}\, \widetilde M_1|_{V_0}\bigr)}{\pi^6};
\]
the coupling part of $-\tr G_2 / \pi^4$ contributes $\tr\bigl(\widetilde B_1|_{V_0, V_0^\perp} D_\perp^{-1} \widetilde B_1|_{V_0^\perp\!, V_0}\bigr)/\pi^4 = \Sigma_\perp(\e)$, and its block part contributes the last two traces of $\Pi_0$. All estimates are uniform on the compact set $S^3$.
\end{proof}

\begin{corollary}[Closed-form Hessian]\label{c:closed-form-hessian}
$H_N = \Hess_\p f_N(\bzero)$ is diagonal in $(e_a, e_b, e_c, e_d)$, with $H_{bb}^N = H_{cc}^N$ and
\[
H_1 = \mathrm{diag}\!\left(\tfrac{2}{\pi^2},\, -\tfrac{32}{\pi^6},\, -\tfrac{32}{\pi^6},\, \tfrac{2}{\pi^2}\right), \qquad
H_2 = \mathrm{diag}\!\left(\tfrac{2}{\pi^2},\, \tfrac{3232}{27\pi^6},\, \tfrac{3232}{27\pi^6},\, \tfrac{2}{\pi^2}\right).
\]
In particular $H_1$ is indefinite, while $H_2 \succ 0$ with $\lambda_{\min}(H_2) = \tfrac{3232}{27\pi^6}$.
\end{corollary}

\begin{proof}
The hypotheses of Lemma~\ref{l:f2-smooth} are open conditions and hold at $\p = \bzero$, where the spectrum of the pencil is $\pi^2$, $\pi^2$, $2\pi^2$, $4\pi^2$, $4\pi^2$; hence $f_2$ is real-analytic on a neighborhood of $\bzero$. (For $N = 1$ the cluster is the entire spectrum of the $2 \times 2$ pencil, its symmetric functions are rational in the entries, and $f_1$ is likewise real-analytic near $\bzero$; the argument below applies verbatim with empty $V_0^\perp$-block and $\Sigma_\perp \equiv 0$.) Combining Lemma~\ref{l:cluster-sum} with $|Q_{t\e}|^{-1} = 1 + t^2 (e_a^2 + e_d^2) + O(t^4)$,
\[
f_N(t\e) \;=\; \frac{1}{2\,|Q_{t\e}|} \Bigl( \frac{1}{\lambda_1(t\e)} + \frac{1}{\lambda_2(t\e)} \Bigr)
\;=\; \frac{1}{\pi^2} \;+\; \frac{t^2}{2} \biggl( \frac{2\,(e_a^2 + e_d^2)}{\pi^2} + \Pi_0(\e) + \Sigma_\perp(\e) \biggr) \;+\; O(t^3),
\]
and comparison with the Taylor expansion of $f_N$ at $\bzero$, whose gradient vanishes (Section~\ref{s:RayleighRitz}), identifies the quadratic form:
\begin{equation}
\label{e:HessianFormula}
\e^\top H_N\, \e \;=\; \frac{2\,(e_a^2 + e_d^2)}{\pi^2} \;+\; \Pi_0(\e) \;+\; \Sigma_\perp(\e).
\end{equation}
By the $D_4$-symmetry argument of Section~\ref{s:RayleighRitz}, $H_N$ is diagonal with $H_{bb}^N = H_{cc}^N$, so it suffices to evaluate \eqref{e:HessianFormula} at $\e = e_a, e_b, e_d$, reading the normalized blocks off Tables~\ref{tab:first-order} and~\ref{tab:second-order-pure}; recall $E = \mathrm{diag}(\sqrt2, \sqrt2, 2, \sqrt2, \sqrt2)$ and $\widetilde K_2 = \tfrac{1}{2} E K^{(2)}_{xx} E$ along $\e = e_x$.

\emph{Direction $a$.} All four matrices are diagonal, with $\widetilde K_1 = \mathrm{diag}(\pi^2, -\pi^2, 0, 4\pi^2, -4\pi^2)$ and $\widetilde M_1 = \mathrm{diag}(-1, 1, 0, -1, 1)$, while $\widetilde K_2 = 0$ and $\widetilde M_2|_{V_0} = \mathrm{diag}(-2, -2)$. In particular $\widetilde B_1$ has no coupling block, so $\Sigma_\perp = 0$, and
\[
\pi^6\, \Pi_0 \;=\; 2\pi^4 \;-\; \pi^2 \cdot (-2\pi^2) \;-\; 0 \;+\; \pi^4 \cdot (-4) \;=\; 0.
\]
Hence $H_{aa}^N = \tfrac{2}{\pi^2}$: the entire contribution comes from the area factor.

\emph{Direction $d$.} Here $\widetilde K_1|_{V_0} = 0$, and $\widetilde M_1|_{V_0}$ is off-diagonal with entry $-\tfrac{16}{\pi^2}$; the remaining nonzero entries of $\widetilde B_1$ pair modes within $V_0^\perp$, so again $\Sigma_\perp = 0$. With $\tr\bigl(\widetilde K_2|_{V_0}\bigr) = -2\pi^2$ and $\tr\bigl(\widetilde M_2|_{V_0}\bigr) = -2$,
\[
\pi^6\, \Pi_0 \;=\; 0 \;-\; 0 \;+\; 2\pi^4 \;-\; 2\pi^4 \;=\; 0,
\qquad \text{so} \qquad H_{dd}^N = \tfrac{2}{\pi^2}.
\]

\emph{Direction $b$.} Now $\widetilde K_1|_{V_0} = \widetilde M_1|_{V_0} = 0$ and $\widetilde K_2 = 0$; the only nonzero cluster-block datum is $\tr\bigl(\widetilde M_2|_{V_0}\bigr) = -\tfrac{32}{\pi^4}$, so $\pi^6 \Pi_0 = -32$. The coupling entries of $\widetilde B_1$ are
\[
\bigl(\widetilde B_1\bigr)_{(1,0), (1,1)} = -8\sqrt2, \qquad
\bigl(\widetilde B_1\bigr)_{(0,1), (2,0)} = -8, \qquad
\bigl(\widetilde B_1\bigr)_{(0,1), (0,2)} = \tfrac{64}{9} - \tfrac{40}{9} = \tfrac{8}{3},
\]
with gap weights $\tfrac{\pi^2}{\kappa_m - \pi^2} = 1, \tfrac{1}{3}, \tfrac{1}{3}$ respectively, so that
\[
\pi^6\, \Sigma_\perp \;=\; 128 + \tfrac{64}{3} + \tfrac{64}{27} \;=\; \tfrac{4096}{27},
\qquad
H_{bb}^2 \;=\; \Pi_0 + \Sigma_\perp \;=\; \frac{-864 + 4096}{27\, \pi^6} \;=\; \frac{3232}{27 \pi^6}.
\]

Finally, for $N = 1$ the cluster blocks (hence $\Pi_0$) are unchanged while $\Sigma_\perp \equiv 0$, giving $H_{bb}^1 = -\tfrac{32}{\pi^6}$ and the same $a, d$ entries.
\end{proof}

\subsection{Local maximality}\label{ss:local-max}

\begin{theorem}\label{t:LocalMin}
For $N = 2$, the point $\p = \bzero$ is a strict local minimum of $f_2$. Consequently, there exists $\rho > 0$ such that the square $\square = Q_\bzero$ is the unique maximizer of $|Q_\p|\mu_1(Q_\p)$ on $\overline{B(\bzero, \rho)} \cap \P$.
\end{theorem}

\begin{proof}
The function $\p \mapsto f_2(\p)$ is real-analytic on a neighborhood of $\p = \bzero$ by Proposition~\ref{l:M3-bound}(b).

By Taylor's theorem, it suffices to verify
 (i) $\nabla f_2(\bzero) = \bzero$, and
(ii) $H_2 := \Hess f_2(\bzero) \succ 0$.
The gradient vanishes (i) by the $D_4$-symmetry argument of Section~\ref{s:RayleighRitz}, and (ii) holds by Corollary~\ref{c:closed-form-hessian}, with $\lambda_{\min}(H_2) = \tfrac{3232}{27\pi^6} > 0$. Combined with \eqref{e:KeyInequality}, this yields the stated conclusion.
\end{proof}

For the global proof of Section~\ref{s:GlobalMax}, it is convenient to have an explicit value for the radius $\rho$ in Theorem~\ref{t:LocalMin}. Rather than transport the central Hessian outward with a bound on the third derivative of $f_2$, or enclose the Hessian field over the ball, we certify the inequality $f_2 > \pi^{-2}$ directly, by a second-order difference-quotient test along rays through the square (Appendix~\ref{app:diff-quotient}). In particular, the closed-form Hessian enters the certified chain only through the explicit number $\rho^\sharp$.

\begin{corollary}\label{t:LocalMinExplicit}
The conclusion of Theorem~\ref{t:LocalMin} holds with the explicit radius $\rho = \rho^\sharp := \lambda_{\min}(H_2) = \tfrac{3232}{27\pi^6} \approx 0.1245$.
\end{corollary}

\begin{proof}
Write $\p = t\e$ with $\e \in S^3$ (the unit sphere in $\mathbb R^4$) and $0 < t \le \rho^\sharp$. By Proposition~\ref{l:M3-bound}, $f_2$ is real-analytic on a neighborhood of $\overline{B(\bzero, \rho^\sharp)}$ and the bottom cluster $\{\lambda_1, \lambda_2\}$ is isolated there. The second-order difference-quotient test of Appendix~\ref{app:diff-quotient} certifies, on a finite cover of $S^3 \times [0, \rho^\sharp]$, the scalar inequality $S(t, \e) > 0$ for all $t > 0$ (Corollary~\ref{cor:single-box-sign-test}, applied box by box), which by \eqref{eq:f2-minus-square-value} is equivalent to $f_2(t\e) > \pi^{-2}$. Hence $f_2(\p) > \tfrac{1}{\pi^2}$ for $0 < \|\p\| \le \rho^\sharp$. Combined with \eqref{e:KeyInequality}, as in the proof of Theorem~\ref{t:LocalMin}, this yields the stated conclusion.
\end{proof}

\subsection*{Numerical realization}
Appendix~\ref{app:diff-quotient},
Table~\ref{tab:local-second-order-certificate} verifies every final box from $13{,}824\times9$ initial
boxes and reaches depth $1$. It certifies
$S\ge 1.010 \times 10^{-2}$,
$\lambda_1\ge 7.630$, and
$\lambda_2\le 13.268$. The code and certificate
data are available in~\cite{SupplementaryCode}.

\section{The square is a global maximizer}
\label{s:GlobalMax}

In this section we establish Step~(II) of the strategy of
Section~\ref{s:RayleighRitz}, completing the proof of Theorem~\ref{t:Main}.
Writing $\lambda_1(\p) \le \cdots \le \lambda_5(\p)$ for the generalized
eigenvalues of the pencil $(K(\p), M(\p))$, we certify
\begin{equation}\label{e:global-target}
  |Q_\p|\,\lambda_1(\p) \;<\; \pi^2 \qquad \text{for all } \p \in \Omega_{\mathrm{II}} := \P_{\mathrm K} \setminus B\left(\bzero, \tfrac{\rho^\sharp}{2}\right).
\end{equation}
The global cover runs to the half-radius sphere $|\p| = \tfrac{\rho^\sharp}{2}$, so that Steps~(I) and~(II) overlap on the collar $\tfrac{\rho^\sharp}{2} \le |\p| \le \rho^\sharp$. 
For $\p \in \P$ with $M(\p) \succ 0$, so that $Q_\p$ is a convex quadrilateral and $V^2_\p \subset H^1_\diamond(Q_\p)$ is five-dimensional,  the Rayleigh--Ritz bound gives
$\lambda_1(\p) \ge \mu_1(Q_\p)$; since $\Omega_{\mathrm{II}} \subset \P_{\mathrm K} \subset \P$, \eqref{e:global-target} implies
$|Q_\p|\mu_1(Q_\p) < \pi^2$ on $\Omega_{\mathrm{II}}$. 

\subsection*{The box test}
The certificate is a direct comparison: enclose the pencil $(K(\p), M(\p))$ over the box, enclose its smallest generalized eigenvalue, and compare with $\pi^2$.

\begin{proposition}[Box eigenvalue bound]\label{p:box-bound}
Let $\mathcal B \subset P_{\mathrm C}$ be a closed axis-aligned box on which $M(\p) \succ 0$, so that the generalized eigenvalues of the pencil are well-defined, and suppose a certified enclosure
\[
  \lambda_1(\p) \in \bigl[\underline\lambda_1(\mathcal B),\, \overline\lambda_1(\mathcal B)\bigr]
  \quad \text{for all } \p \in \mathcal B,
  \qquad \underline\lambda_1(\mathcal B) > 0,
\]
is available. Write $Q(\mathcal B) := \sup_{\p \in \mathcal B} |Q_\p|$. If
\[
  Q(\mathcal B)\, \overline\lambda_1(\mathcal B) \;<\; \pi^2,
\]
then $0 < |Q_\p|\,\lambda_1(\p) < \pi^2$ on $\mathcal B$.
\end{proposition}

\begin{proof}
On $P_{\mathrm C}$ we have $|Q_\p| = 1 - a^2 - d^2 \ge 1 - \alpha^2 - \delta^2 > 0$ by \eqref{e:cube}, so for every $\p \in \mathcal B$,
$0 < |Q_\p|\,\lambda_1(\p) \;\le\; Q(\mathcal B)\, \overline\lambda_1(\mathcal B) \;<\; \pi^2$.
\end{proof}

No assumption $\mathcal B \subset \P$ is made (retained boxes may straddle $\partial\P$); the comparison with $\mu_1(Q_\p)$ is invoked only at $\p \in \P$.

\subsection*{Certified enclosures across a box}
The bounds are certified in interval arithmetic. The matrices $K(\p), M(\p)$ are
enclosed over $\mathcal B$ as interval matrices $K(\mathcal B), M(\mathcal B)$
containing $\{K(\p) \colon \p \in \mathcal B\}$ and $\{M(\p) \colon \p \in \mathcal B\}$,
by evaluating the integral representation of Section~\ref{app:integral-rep} with
the box as interval argument (Appendix~\ref{app:implementation}). The area factor is enclosed by 
\[
  Q(\mathcal B) = 1 - \underline{a^2}_{\mathcal B} - \underline{d^2}_{\mathcal B},
\]
where $\underline{\,\cdot\,}$ denotes the certified lower interval bound over
$\mathcal B$, i.e.\
$\underline{a^2}_{\mathcal B} := \inf_{\p \in \mathcal B} a^2$,
$\underline{d^2}_{\mathcal B} := \inf_{\p \in \mathcal B} d^2$ (so that
$Q(\mathcal B) = \sup_{\p \in \mathcal B} |Q_\p|$ by \eqref{e:Area}),
each a few interval operations.
Positive definiteness $M(\mathcal B)\succ0$ is certified
by an interval $LDL^\top$ factorization.

\subsection*{Box-covering algorithm}
The proof of \eqref{e:global-target} on $\Omega_{\mathrm{II}}$ is reduced to the
following procedure, the \emph{adaptive box-covering algorithm}.

\begin{enumerate}
\item[(II.a)] \emph{Initial covering.} Cover the outer-bound cube
$P_{\mathrm{C}}$ of \eqref{e:cube}, which contains $\P_{\mathrm{K}}$ by
Lemma~\ref{l:PK-bounded}, by uniform axis-aligned boxes. Discard any box lying
entirely inside $B\left(\bzero, \tfrac{\rho^\sharp}{2}\right)$ (covered by Step~(I)) or entirely
outside $\P_{\mathrm{K}}$, both decidable from the explicit polynomial
inequalities defining these regions. Use the $D_4$-symmetry of the spectrum
(see \eqref{e:equivariance}) to retain one representative per orbit,
reducing the box count by approximately a factor of $8$.

\item[(II.b)] \emph{Box-level test.} For a retained box $\mathcal B_k$ with
center $\p_k$, certify $M(\mathcal B_k) \succ 0$ and call \texttt{veigs} on
the full interval pencil. If the returned index range contains $1$ and the
certified upper bound $\overline\lambda_1(\mathcal B_k)$ satisfies
$Q(\mathcal B_k)\overline\lambda_1(\mathcal B_k)<\pi^2$, the box is verified.

\item[(II.c)] \emph{Adaptive subdivision.} Otherwise bisect $\mathcal B_k$ and
re-apply (II.b). The bisection may use any rule under which repeated
subdivision drives every box's diameter to zero; longest-side bisection is the
simplest such rule, and the implementation of Appendix~\ref{app:implementation}
uses a slack-driven coordinate choice with a longest-side fallback that
guarantees this property.
\end{enumerate}

\begin{theorem}[Finite-cover certificate]\label{t:box-cover-terminates}
If the adaptive box-covering algorithm terminates with every box verified or
discarded, then
$|Q_\p|\mu_1(Q_\p) < \pi^2$ for all $\p \in \Omega_{\mathrm{II}}$; combined with
Step~(I) this establishes \eqref{e:f2-target} and hence Theorem~\ref{t:Main}.
\end{theorem}

\begin{proof}
The boxes of the subdivision, together with
their $D_4$-images, cover $P_{\mathrm{C}} \supset \Omega_{\mathrm{II}}$, and each
is verified or discarded (the retained boxes are one representative per
$D_4$-orbit, and both the spectrum (by \eqref{e:equivariance}) and $\P_{\mathrm{K}}$ (Section~\ref{s:RayleighRitz}) are $D_4$-invariant, so
verifying a representative verifies its orbit). On a verified box $\mathcal B$, Proposition~\ref{p:box-bound} gives
$|Q_\p|\lambda_1(\p) < \pi^2$ for every $\p \in \mathcal B$; for
$\p \in \mathcal B \cap \Omega_{\mathrm{II}} \subset \P$ the Rayleigh--Ritz
comparison of Section~\ref{s:GlobalMax} applies, so
$|Q_\p|\mu_1(Q_\p) \le |Q_\p|\lambda_1(\p) < \pi^2$ there. Discarded boxes either lie
outside $\P_{\mathrm{K}}$, hence are disjoint from $\Omega_{\mathrm{II}}$, or lie
in $\overline{B\left(\bzero, \tfrac{\rho^\sharp}{2}\right)}$, where $|Q_\p|\mu_1(Q_\p) < \pi^2$ away
from $\bzero$ by Step~(I) (Corollary~\ref{t:LocalMinExplicit}). Hence
$|Q_\p|\mu_1(Q_\p) < \pi^2$ on $\Omega_{\mathrm{II}}$.
\end{proof}

\subsection*{Numerical realization}
The global run starts from $16$ boxes, verifies $114{,}627$, discards
$15{,}623$, performs $130{,}234$ bisections, and reaches depth $30$. Here,
\mbox{$\Delta_*\ge2.654\times10^{-5}$}; see
Appendix~\ref{app:implementation}. The
code and certificate data are available in~\cite{SupplementaryCode}.

\begin{proof}[Completion of the proof of Theorem~\ref{t:Main}]
The certified run terminates with every box verified or discarded, so
Theorem~\ref{t:box-cover-terminates} gives $|Q_\p|\mu_1(Q_\p) < \pi^2$ on
$\Omega_{\mathrm{II}}$; combined with Step~(I)
(Corollary~\ref{t:LocalMinExplicit}), this establishes \eqref{e:f2-target} and
hence Theorem~\ref{t:Main}.
\end{proof}

\section*{Acknowledgments}
This work was initiated at the American Institute of Mathematics (AIM) workshop \emph{Symmetry-breaking of optimal shapes} (Pasadena, CA, June 17--21, 2024). The authors thank Dorin Bucur and Lukas Bundrock for helpful discussions.


\printbibliography

\clearpage 
\appendix
\section*{Appendices}
\noindent Appendix~\ref{app:diff-quotient} gives the certified difference-quotient test on the local ball and the cluster margins, behind Proposition~\ref{l:M3-bound} and Corollary~\ref{t:LocalMinExplicit}; and Appendix~\ref{app:implementation} documents the per-box test of Section~\ref{s:GlobalMax}.

\section{Certified computation for the local step}\label{app:diff-quotient}

\begingroup
\fontsize{10pt}{11pt}\selectfont
\setlength{\abovedisplayskip}{6pt plus 2pt minus 1pt}
\setlength{\belowdisplayskip}{6pt plus 2pt minus 1pt}
\setlength{\abovedisplayshortskip}{3pt plus 1pt}
\setlength{\belowdisplayshortskip}{4pt plus 1pt minus 1pt}
\setlength{\jot}{2pt}
\setlength{\intextsep}{8pt plus 2pt minus 2pt}

By \eqref{e:KeyInequality}, Step~(I) follows from
\begin{equation}\label{eq:local-statement}
        f_2(\p) > \pi^{-2}
        \qquad (0 < \|\p\| \le \rho^\sharp),
        \qquad
        \rho^\sharp:=\frac{3232}{27\pi^6}.
\end{equation}
We prove \eqref{eq:local-statement} by reducing it to
\eqref{eq:f2-minus-square-value} and verifying its numerator by
interval arithmetic.

\subsection{The inequality to be certified}

Let $\p:=(a,b,c,d)\in\P$ be as in Section~\ref{s:QuadParam}, so that
$Q_{\bzero}=\square$.  For $\p\ne\bzero$, define
$S^3:=\{\e\in\mathbb R^4:\|\e\|=1\}$, $t:=\|\p\|$, and
$\e:=\p/\|\p\|\in S^3$, so that $\p=t\e$.
The function $f_2$ and the generalized
eigenvalues $\lambda_1\le\lambda_2$ of the pencil
$(K(\p),M(\p))$, with $K(\p),M(\p)\in\mathbb R^{5\times5}$, are
defined in Sections~\ref{s:RayleighRitz} and~\ref{app:integral-rep}.  Since
$\lambda_1(\bzero)=\lambda_2(\bzero)=\pi^2$,
we use symmetric functions of this eigenvalue pair.  For
$\e\in S^3$ and $0<t\le\rho^\sharp$, set
\begin{equation}\label{eq:nu-L-definitions}
    \nu_i(t,\e)
        :=\frac{\lambda_i(t\e)-\pi^2}{t}\quad(i=1,2),\qquad
    L(t,\e)
        :=\frac{\nu_1(t,\e)+\nu_2(t,\e)}{t}
          =\frac{\lambda_1(t\e)+\lambda_2(t\e)-2\pi^2}{t^2}.
\end{equation}
For the pair $\lambda_1,\lambda_2$, write
$\Sigma(\p):=\lambda_1(\p)+\lambda_2(\p)$ and
$G(\p):=(\lambda_1(\p)-\pi^2)(\lambda_2(\p)-\pi^2)$.
These symmetric functions are smooth in a neighborhood of $\p=\bzero$.
At $\p=\bzero$, equation \eqref{e:f2-via-Sigma-Pi}, together with
$D|Q_\p|_{\p=\bzero}=0$ and $\nabla f_2(\bzero)=0$, gives
$\Sigma(\bzero)=2\pi^2$, $D\Sigma(\bzero)=0$, $G(\bzero)=0$, and $DG(\bzero)=0$;
indeed, $D(\lambda_1\lambda_2)(\bzero)=\pi^2D\Sigma(\bzero)$ and
$DG(\bzero)=D(\lambda_1\lambda_2)(\bzero)-\pi^2D\Sigma(\bzero)=0$.
Taylor's formula therefore yields the continuous extensions
\[
 L(t,\e)
 =\int_0^1(1-s)D^2\Sigma(st\e)[\e,\e]\,ds,
 \qquad
 \nu_1(t,\e)\nu_2(t,\e)
 =\int_0^1(1-s)D^2G(st\e)[\e,\e]\,ds
 \quad (t\ge0).
\]
Define
\begin{equation}\label{eq:S-definition}
    S(t,\e)
    :=-\pi^2L(t,\e)-2\nu_1(t,\e)\nu_2(t,\e)
      +2(e_a^2+e_d^2)
       \left\{\pi^4+t^2\bigl(\pi^2L(t,\e)
       +\nu_1(t,\e)\nu_2(t,\e)\bigr)\right\}.
\end{equation}
Equations \eqref{e:Area}, \eqref{eq:nu-L-definitions}, and
\eqref{eq:S-definition} give
\begin{equation}\label{eq:f2-minus-square-value}
    f_2(t\e)-\pi^{-2}
    =\frac{t^2S(t,\e)}
    {2\pi^2|Q_{t\e}|\lambda_1(t\e)\lambda_2(t\e)}.
\end{equation}
Since $e_a^2+e_d^2\le1$, we have
\begin{equation}\label{eq:area-positive-local-ball}
 |Q_{t\e}|=1-t^2(e_a^2+e_d^2)
 \ge 1-(\rho^\sharp)^2>0.
\end{equation}
By the estimate $\min_i c_i(\p)>0.69$ on
$\overline{B(\bzero,\rho^\sharp)}$ (see the proof of
Proposition~\ref{l:M3-bound} (b)),
\[
\overline{B(\bzero,\rho^\sharp)}\subset\operatorname{int}\P.
\]
Corollary~\ref{cor:local-cluster-gap} also gives
$M(\p)\succ0$ and $K(\p)\succ0$ on
$\overline{B(\bzero,\rho^\sharp)}$, hence
$\lambda_1(\p)\lambda_2(\p)>0$.
Thus \eqref{eq:area-positive-local-ball}
and \eqref{eq:f2-minus-square-value} reduce
\eqref{eq:local-statement} to
\begin{equation}\label{eq:S-positive-target}
 S(t,\e)>0
 \qquad
 (\e,t)\in S^3\times(0,\rho^\sharp].
\end{equation}

\begin{remark}
The construction extends the first-order difference quotients used in
\cite{endoLiu2026partII} for Dirichlet eigenvalues of triangles in a
neighborhood of an equilateral triangle.  At the square, the double
eigenvalue requires the quotients in \eqref{eq:nu-L-definitions};
\eqref{eq:Ft-definition} gives their bounds without third derivatives
of $f_2$.
\end{remark}

\subsection{The Schur complement reduction}

In the ordered basis
$((1,0),(0,1),(1,1),(2,0),(0,2))$, define the columns of
$E=M(\bzero)^{-1/2}$ in Section~\ref{ss:cluster-sum} by
\[
 E_0:=
 \begin{bmatrix}
  \sqrt2&0\\0&\sqrt2\\0&0\\0&0\\0&0
 \end{bmatrix},
 \qquad
 E_\perp:=
 \begin{bmatrix}
  0&0&0\\0&0&0\\2&0&0\\0&\sqrt2&0\\0&0&\sqrt2
 \end{bmatrix}.
\]
At $\p=\bzero$, we have
\[
    M(\bzero)
      =\operatorname{diag}\left(\tfrac12,\tfrac12,\tfrac14,
                                \tfrac12,\tfrac12\right),\qquad
    E_0^\top M(\bzero)E_0=I_2,\qquad
    E_\perp^\top M(\bzero)E_\perp=I_3,\qquad
    E_0^\top M(\bzero)E_\perp=0.
\]
For $X\in\mathbb R^{5\times5}$, write
\[
    X_{00}:=E_0^\top XE_0,\qquad
    X_{0\perp}:=E_0^\top XE_\perp,\qquad
    X_{\perp0}:=E_\perp^\top XE_0,\qquad
    X_{\perp\perp}:=E_\perp^\top XE_\perp.
\]
The four blocks of $K(\bzero)-\pi^2M(\bzero)$ vanish except for
$D_\perp:=\bigl(K(\bzero)-\pi^2M(\bzero)\bigr)_{\perp\perp}
=\operatorname{diag}(\pi^2,3\pi^2,3\pi^2)$, as in
Section~\ref{ss:cluster-sum}.

For $\e\in S^3$, $\nu\in\mathbb R$, and $t\ge0$, define
\begin{equation}\label{eq:Ct-definition}
    C_t(\nu)
    :=\int_0^1D(K-\pi^2M)(st\e)[\e]\,ds-\nu M(t\e).
\end{equation}
In particular, we have
\[
    C_0(\nu)=D(K-\pi^2M)(\bzero)[\e]-\nu M(\bzero),\qquad
    K(t\e)-(\pi^2+t\nu)M(t\e)
    =K(\bzero)-\pi^2M(\bzero)+tC_t(\nu).
\]
Next, define
\begin{equation}\label{eq:Ft-definition}
    \Phi_t(\nu,\e)
      :=C_{t,0\perp}(\nu)
        \bigl(D_\perp+tC_{t,\perp\perp}(\nu)\bigr)^{-1}
        C_{t,\perp0}(\nu),\qquad
    F_t(\nu):=C_{t,00}(\nu)-t\Phi_t(\nu,\e).
\end{equation}
By \eqref{eq:Ct-definition}, we have
\begin{equation}\label{eq:schur-block-matrix}
\begin{bmatrix}E_0^\top\\E_\perp^\top\end{bmatrix}
 \bigl(K(t\e)-(\pi^2+t\nu)M(t\e)\bigr)
 \begin{bmatrix}E_0&E_\perp\end{bmatrix}
=
\begin{bmatrix}
 tC_{t,00}(\nu)&tC_{t,0\perp}(\nu)\\
 tC_{t,\perp0}(\nu)&D_\perp+tC_{t,\perp\perp}(\nu)
\end{bmatrix}.
\end{equation}
If $t>0$ and $D_\perp+tC_{t,\perp\perp}(\nu)$ is invertible, then
\eqref{eq:schur-block-matrix} and \eqref{eq:Ft-definition} identify
$tF_t(\nu)$ with the Schur complement
$B^{\mathrm{eff}}(\pi^2+t\nu;t\e)$ of
Lemma~\ref{l:reduced-resolvent}, applied with the fixed frames
$V=E_0$, $W=E_\perp$; here the invertibility of the
$\perp$-block is certified on each box in
Section~\ref{sec:single-box-test} rather than by continuity.
The determinant identity in the proof of
Lemma~\ref{l:reduced-resolvent} then gives
\begin{equation}\label{eq:schur-equivalence}
    \det F_t(\nu)=0
    \quad\text{if and only if there exists }x\ne0\text{ such that}\quad
    K(t\e)x=(\pi^2+t\nu)M(t\e)x.
\end{equation}

\subsection{Interval conditions}
\label{sec:single-box-test}

For
$\underline e_r\le\overline e_r$,
$0\le\underline t\le\overline t\le\rho^\sharp$, and
$\nu_-<\nu_+$, define
\[
 [\e]:=\prod_{r=1}^4[\underline e_r,\overline e_r],\qquad
 [t]:=[\underline t,\overline t],\qquad
 [\nu]:=[\nu_-,\nu_+],\qquad
 \mathcal B:=([\e]\cap S^3)\times[t].
\]
The required inclusion is
$\bigl\{\nu_i(t,\e):(\e,t)\in\mathcal B,\ t>0,\ i=1,2\bigr\}
\subset[\nu]$.
For this box, define
\[
    \widehat F_t(\nu)
      :=F_t(0)+\nu\,\partial_\nu F_t(0),\qquad
    \Psi_t(\nu,\e)
      :=\Phi_t(\nu,\e)-\Phi_t(0,\e)
        -\nu\,\partial_\nu\Phi_t(0,\e).
\]
Write
$\det\widehat F_t(\nu)=d_2(t,\e)\nu^2-b_1(t,\e)\nu+d_0(t,\e)$.
At $t=0$, we have
$\widehat F_0(\nu)
=\bigl(DK(\bzero)[\e]-\pi^2DM(\bzero)[\e]\bigr)_{00}-\nu I_2$.
Since $\det(A-\nu I_2)=\nu^2-\operatorname{tr}(A)\nu+\det A$,
\eqref{e:f2-via-Sigma-Pi} and
$D|Q_\p|_{\p=\bzero}=0$ give
\[
 b_1(0,\e)
 =\operatorname{tr}
   \bigl(DK(\bzero)[\e]-\pi^2DM(\bzero)[\e]\bigr)_{00}
 =D(\lambda_1+\lambda_2)(\bzero)[\e]
 =-2\pi^4Df_2(\bzero)[\e]
 =0.
\]
Define
\begin{equation}\label{eq:d1-continuous-extension}
    \widetilde d_1(t,\e)
       :=\int_0^1\partial_t b_1(st,\e)\,ds .
\end{equation}
Thus $b_1(t,\e)=t\widetilde d_1(t,\e)$ and
\begin{equation}\label{eq:Fhat-determinant}
    \det\widehat F_t(\nu)
       =d_2(t,\e)\nu^2-t\widetilde d_1(t,\e)\nu+d_0(t,\e),
\end{equation}
so \eqref{eq:Fhat-determinant} remains an interval formula when
$0\in[t]$.  Since $C_{t,00}$ is affine in $\nu$, we have
\begin{equation}\label{eq:Fhat-error}
    F_t(\nu)-\widehat F_t(\nu)=-t\Psi_t(\nu,\e).
\end{equation}

The following conditions imply the bounds required in
\eqref{eq:S-definition}.
\begin{theorem}[Enclosures on $\mathcal B$]
\label{thm:single-box-enclosures}
Let $\eta\ge0$ and $\beta>0$ be fixed, and suppose that, for every $(\e,t)\in\mathcal B$ and every $\nu\in[\nu]$, we have
\begin{equation}\label{eq:root-inclusion-test}
    M(t\e) \succ 0, \quad  
    D_\perp+tC_{t,\perp\perp}(\nu) \succ 0, \quad
    F_t(\nu_-) \succ 0, \quad 
    F_t(\nu_+) \prec 0, \quad 
    \partial_\nu F_t(\nu) \prec 0
\end{equation}
and
\begin{equation}\label{eq:Fhat-comparison-test}
    \|\Psi_t(\nu,\e)\|_2 \le \eta, \qquad -\partial_\nu\widehat F_t(\nu) \succeq \beta I_2.
\end{equation}
Then $\nu_1(t,\e),\nu_2(t,\e)\in[\nu]$ for every
$(\e,t)\in\mathcal B$ with $t>0$, and
\begin{equation}\label{eq:symmetric-single-box-enclosures}
\begin{aligned}
    L(t,\e)
      &\in \frac{\widetilde d_1}{d_2}
            +\frac{2\eta}{\beta}[-1,1],\\
    \nu_1(t,\e)\nu_2(t,\e)
      &\in \frac{d_0}{d_2}
            +\left(
              2\max\{|\nu_-|,|\nu_+|\}\frac{t\eta}{\beta}
              +\frac{t^2\eta^2}{\beta^2}
             \right)[-1,1].
\end{aligned}
\end{equation}
\end{theorem}

\begin{proof}
Fix $(\e,t)\in\mathcal B$, $t>0$.  Equations
\eqref{eq:schur-block-matrix} and \eqref{eq:Ft-definition} give
\[
 K(t\e)-(\pi^2+t\nu)M(t\e)
 \ \sim\
 \begin{bmatrix}
  tF_t(\nu)&0\\
  0&D_\perp+tC_{t,\perp\perp}(\nu)
 \end{bmatrix},
\]
where $\sim$ denotes congruence.  The min--max principle and
\eqref{eq:root-inclusion-test} imply
that if $F_t(\nu_-)\succ0$, then $\pi^2+t\nu_-<\lambda_1(t\e)$, and
if $F_t(\nu_+)\prec0$, then
$\lambda_2(t\e)<\pi^2+t\nu_+<\lambda_3(t\e)$.
Hence $\nu_i(t,\e)\in[\nu_-,\nu_+]$, $i=1,2$.

Write $\widehat F_t(\nu)=A_t-\nu B_t$ and let
$\widehat\nu_1\le\widehat\nu_2$ be its generalized eigenvalues.
By \eqref{eq:Fhat-comparison-test}, we have
$B_t\succeq\beta I_2$ and $d_2=\det B_t\ge\beta^2>0$, and
Vieta's formulas applied to \eqref{eq:Fhat-determinant} give
$\widehat\nu_1+\widehat\nu_2=t\widetilde d_1/d_2$ and
$\widehat\nu_1\widehat\nu_2=d_0/d_2$.

For $\|x\|=1$, define $r(x)$ and $\widehat r(x)$ by
$x^\top F_t(r(x))x=0$ and
$x^\top\widehat F_t(\widehat r(x))x=0$.
Existence and uniqueness follow from
\eqref{eq:root-inclusion-test} and
\eqref{eq:Fhat-comparison-test}.  Moreover,
\eqref{eq:schur-equivalence} and the min--max principle give
\begin{equation}\label{eq:rayleigh-root-formulas}
 \nu_1=\min_{\|x\|=1}r(x),\qquad
 \nu_2=\max_{\|x\|=1}r(x),\qquad
 \widehat\nu_1=\min_{\|x\|=1}\widehat r(x),\qquad
 \widehat\nu_2=\max_{\|x\|=1}\widehat r(x).
\end{equation}
By \eqref{eq:Fhat-error}, we have
\[
 \beta|\widehat r(x)-r(x)|
 \le
 \left|x^\top
 \bigl(\widehat F_t(r(x))-\widehat F_t(\widehat r(x))\bigr)x\right|
 =\left|x^\top
 \bigl(\widehat F_t(r(x))-F_t(r(x))\bigr)x\right|
 \le t\eta .
\]
Together with \eqref{eq:rayleigh-root-formulas}, this yields
$|\nu_i-\widehat\nu_i|\le t\eta/\beta$, $i=1,2$.
Therefore, we have
\[
 \left|L-\frac{\widetilde d_1}{d_2}\right|
 \le\frac{2\eta}{\beta},\qquad
 \left|\nu_1\nu_2-\frac{d_0}{d_2}\right|
 \le
 2\max\{|\nu_-|,|\nu_+|\}\frac{t\eta}{\beta}
 +\frac{t^2\eta^2}{\beta^2},
\]
which proves \eqref{eq:symmetric-single-box-enclosures}.
\end{proof}

For fixed $(\e,t)\in\mathcal B$ with $t>0$, let
$\widehat\nu_1\le\widehat\nu_2$ be the zeros of
$\det\widehat F_t$, and define
$E_t(\nu):=\widehat F_t(\nu)-F_t(\nu)=t\Psi_t(\nu,\e)$
and $\delta_i:=\nu_i-\widehat\nu_i$, $i=1,2$.
For disjoint intervals containing the two roots, the following identity
improves \eqref{eq:symmetric-single-box-enclosures}.
\begin{lemma}\label{lem:exact-root-refinement}
For $\{i,j\}=\{1,2\}$, we have
\begin{equation}\label{eq:exact-root-refinement}
 \nu_i-\widehat\nu_i
 =
 \frac{
 \operatorname{tr}\!\left(
   \operatorname{adj}\widehat F_t(\nu_i)\,E_t(\nu_i)\right)
 -\det E_t(\nu_i)}
 {d_2(\nu_i-\widehat\nu_j)},
\end{equation}
where $\operatorname{adj} A$ denotes the adjugate of the matrix $A$.

If $[\nu_i]$ contains $\nu_i$ and $\widehat\nu_i$ and
$[\nu_1]\cap[\nu_2]=\varnothing$, the denominator in
\eqref{eq:exact-root-refinement} does not contain zero, and
\begin{equation}\label{eq:refined-symmetric-quantities}
 L
 =\frac{\widetilde d_1}{d_2}
   +\frac{\delta_1+\delta_2}{t},\qquad
 \nu_1\nu_2
 =\frac{d_0}{d_2}
   +\widehat\nu_1\delta_2+\widehat\nu_2\delta_1
   +\delta_1\delta_2 .
\end{equation}
\end{lemma}

\begin{proof}
\[
0
=\det F_t(\nu_i)
=\det\widehat F_t(\nu_i)
  -\operatorname{tr}\!\left(
    \operatorname{adj}\widehat F_t(\nu_i)E_t(\nu_i)\right)
  +\det E_t(\nu_i),
\qquad
 \det\widehat F_t(\nu)
=d_2(\nu-\widehat\nu_1)(\nu-\widehat\nu_2).
\]
These identities give \eqref{eq:exact-root-refinement}; substitution
of $\nu_i=\widehat\nu_i+\delta_i$ into Vieta's formulas gives
\eqref{eq:refined-symmetric-quantities}.
\end{proof}

The inverse in \eqref{eq:Ft-definition} gives a matrix
$R_t(\nu)$ satisfying $E_t(\nu)=t^3\nu^2R_t(\nu)$, and hence
\[
 \frac{\delta_i}{t}
 =
 \frac{
 \operatorname{tr}\!\left(
  \operatorname{adj}\widehat F_t(\nu_i)\,E_t(\nu_i)/t\right)
 -\det E_t(\nu_i)/t}
 {d_2(\nu_i-\widehat\nu_j)}.
\]
Moreover, since $E_t(\nu)/t=t^2\nu^2R_t(\nu)$ and
$\det E_t(\nu)/t=t^5\nu^4\det R_t(\nu)$,
the quotient extends continuously to $t=0$ without division by
$[t]$.
If $\widehat\nu_j\in[\nu_j]$, then
$\nu_i-\widehat\nu_j\in[\nu_i]-[\nu_j]$.
Substitution in \eqref{eq:S-definition} gives the required sign.
\begin{corollary}
\label{cor:single-box-sign-test}
Under the hypotheses of Theorem~\ref{thm:single-box-enclosures}, suppose
that the enclosures in
\eqref{eq:symmetric-single-box-enclosures} or their intersections with
\eqref{eq:refined-symmetric-quantities}, substituted in
\eqref{eq:S-definition}, give $S(t,\e)>0$
throughout $\mathcal B$.  Then
$f_2(t\e)>\pi^{-2}$ for every $(\e,t)\in\mathcal B$ with $t>0$.
\end{corollary}

\begin{proof}
This follows immediately from \eqref{eq:f2-minus-square-value}.
\end{proof}

\subsection{Finite cover and certification}
\label{subsec:single-box-algorithm}

Let $u_1,\ldots,u_4$ be the standard basis of $\mathbb R^4$.  For
$r\in\{1,\ldots,4\}$, $\sigma\in\{-1,1\}$, and
$\{j_1,j_2,j_3\}=\{1,\ldots,4\}\setminus\{r\}$, define
\begin{equation}\label{eq:sphere-chart}
 \chi_{r,\sigma}(x)
 :=
 \frac{\sigma u_r+\sum_{\ell=1}^3x_\ell u_{j_\ell}}
 {\sqrt{1+\|x\|^2}},
 \qquad x\in[-1,1]^3.
\end{equation}
For every $\e\in S^3$, choose $r\in\arg\max_j|e_j|$ and set
$\sigma=\operatorname{sgn}(e_r)$ and
$x_\ell=e_{j_\ell}/|e_r|$.  Then $x_\ell\in[-1,1]$ and
$\e=\chi_{r,\sigma}(x)$.
Hence the eight maps in \eqref{eq:sphere-chart} cover $S^3$.  Set
$I_k:=\bigl[-1+\tfrac{k}{6},\,-1+\tfrac{k+1}{6}\bigr]$,
$k=0,\ldots,11$.
Interval evaluation of \eqref{eq:sphere-chart} on
$I_{k_1}\times I_{k_2}\times I_{k_3}$ gives
$8\cdot12^3=13{,}824$
boxes in the $\e$ coordinates.  In the $t$ coordinate, set
\begin{equation}\label{eq:radial-grid}
 t_k:=\frac{\rho^+}{16}a_k,\qquad
 (a_0,\ldots,a_9):=(0,2,4,6,8,10,12,14,15,16),
\end{equation}
where $\rho^+$ is a certified outward-rounded upper bound for
$(1+10^{-12})\rho^\sharp$, and hence $\rho^+\ge\rho^\sharp$.
The last two radial intervals have width $\rho^+/16$ rather than
$\rho^+/8$, refining the cover in the neighborhood of $t=\rho^\sharp$. 
Equations \eqref{eq:sphere-chart}--\eqref{eq:radial-grid} therefore cover $S^3\times[0,\rho^\sharp]$.

Algorithm~\ref{alg:single-box-certificate} is applied to every product
of one of these boxes and $[t_k,t_{k+1}]$.  An undecided product
$I_{k_1}\times I_{k_2}\times I_{k_3}$ is replaced by its eight
midpoint subdivisions.  For every child $\mathcal B'$ of
$\mathcal B$, we have
\begin{equation}\label{eq:subdivision-inclusion}
 \mathcal B'\subset\mathcal B,\qquad
 [g]_{\mathcal B'}\subset[g]_{\mathcal B}
\end{equation}
for every interval expression $[g]$ evaluated on the parent.

\begin{algorithm}[p]
\footnotesize
\caption{Certificate on $\mathcal B$}
\label{alg:single-box-certificate}
\begin{algorithmic}[1]
\Require $[\e]$ and $[t]$; set
\(\mathcal B:=([\e]\cap S^3)\times[t]\).
\Ensure \textsc{verified} or \textsc{undecided}.

\State Enclose \(K(t\e)\), \(M(t\e)\), and the blocks of \(C_t(\nu)\)
from \eqref{eq:Ct-definition}.
\State Set
\[
 B_0:=D_\perp+tC_{t,\perp\perp}(0),\qquad
 m_M:=\min_{1\le r\le5}
 \left(
   \inf[M_{rr}]_{\mathcal B}
   -\sum_{s\ne r}\sup|[M_{rs}]_{\mathcal B}|
 \right).
\]
\If{\(m_M\le0\) or \([B_0]_{\mathcal B}\succ0\) cannot be certified}
  \State Return \textsc{undecided}.
\EndIf

\State Form \([B_0]_{\mathcal B}^{-1}\) and enclose
\(F_t(0)\), \(\partial_\nu F_t(0)\), \(d_2\),
\(\widetilde d_1\), and \(d_0\) using
\eqref{eq:Ft-definition}--\eqref{eq:Fhat-determinant}.
\State Set
\[
 \widehat F_t(\nu):=F_t(0)+\nu\,\partial_\nu F_t(0),
 \qquad
 \beta:=\inf\lambda_{\min}
 \bigl(-[\partial_\nu\widehat F_t]_{\mathcal B}\bigr).
\]
\If{\(\inf[d_2]_{\mathcal B}\le0\) or \(\beta\le0\)}
  \State Return \textsc{undecided}.
\EndIf

\State Using \eqref{eq:Fhat-determinant}, set
\[
 [\Delta]
 :=\max\!\left\{
 [t]^2[\widetilde d_1]^2-4[d_2][d_0],\,0
 \right\},
 \qquad
 [\widehat\nu_{1,2}]
 :=\frac{[t][\widetilde d_1]\mp\sqrt{[\Delta]}}
 {2[d_2]}.
\]
\State Starting with
\(\operatorname{hull}([\widehat\nu_1]\cup[\widehat\nu_2])\),
choose \([\nu]=[\nu_-,\nu_+]\) and a finite subdivision
\([\nu]=\bigcup_W W\).
\State On every \(W\), first certify
\[
 [D_\perp+tC_{t,\perp\perp}(\nu)]_{\mathcal B\times W}\succ0,
\]
and only then form its interval inverse and enclose
\(\Psi_t\) and \(\partial_\nu F_t\) using
\eqref{eq:Ft-definition} and \eqref{eq:Fhat-error}.
\State Choose \(\eta\ge0\) such that
\[
 \eta\ge
 \max_W\sup\|[\Psi_t]_{\mathcal B\times W}\|_2,
 \qquad
 [\widehat\nu_i]
 +\frac{[t]\eta}{\beta}[-1,1]
 \subset[\nu],
 \quad i=1,2.
\]
\If{the preceding inclusions or any condition in
\eqref{eq:root-inclusion-test}--\eqref{eq:Fhat-comparison-test}
cannot be certified}
  \State Return \textsc{undecided}.
\EndIf

\State Set
\[
 [\nu_i]
 :=[\widehat\nu_i]+\frac{[t]\eta}{\beta}[-1,1],
 \qquad i=1,2,
\]
and form the enclosures in
\eqref{eq:symmetric-single-box-enclosures}.
\If{\([\nu_1]\cap[\nu_2]=\varnothing\)}
  \State Intersect these enclosures with those obtained from
  \eqref{eq:exact-root-refinement}--%
  \eqref{eq:refined-symmetric-quantities}.
\EndIf

\State Substitute the resulting enclosures for
\(L\) and \(\nu_1\nu_2\) into \eqref{eq:S-definition} to obtain
\([S]_{\mathcal B}\), and set
\[
 [\lambda_i]:=\pi^2+[t][\nu_i],
 \qquad i=1,2.
\]
\State Use \texttt{veigs} to obtain verified enclosures
\([\lambda_1]^{\mathrm v}\) and \([\lambda_3]^{\mathrm v}\)
for indices \(1\) and \(3\) of
\((K(t\e),M(t\e))\), and set
\[
 [\lambda_1]\leftarrow
 [\lambda_1]\cap[\lambda_1]^{\mathrm v}.
\]
\If{this intersection is empty or
\(\inf[\lambda_3]^{\mathrm v}\le16\)}
  \State Return \textsc{undecided}.
\EndIf
\State Set
\[
 m_3:=\inf[\lambda_3]^{\mathrm v}-\frac{3\pi^2}{2}.
\]

\If{\(\inf[S]_{\mathcal B}>0\) and \(\inf[\lambda_1]>0\)}
  \State Return \textsc{verified} by
  Corollary~\ref{cor:single-box-sign-test}, with certified bounds
  \(m_M\) and \(m_3\).
\Else
  \State Return \textsc{undecided}.
\EndIf
\end{algorithmic}
\end{algorithm}

\begin{table}[t]
\centering
\renewcommand{\arraystretch}{1.05}
\begin{tabular}{p{0.43\linewidth}p{0.47\linewidth}}
\hline
Quantity & Certified value \\
\hline
Number of Boxes & \(13{,}824\times9\) \\
Certified lower bound for \(S\) &
\(1.010\times10^{-2}\) \\
Certified lower bound for \(\lambda_1\) &
\(7.630\) \\
Certified upper bound for \(\lambda_2\) &
\(13.268\) \\
Certified lower bound for
$\lambda_{\min}(M)$ & $0.0521$ \\
Certified lower bound for
$\lambda_3-3\pi^2/2$ & $3.030$ \\
Outcome & \textsc{verified} on every final box \\
\hline
\end{tabular}
\caption{Certified results for
Algorithm~\ref{alg:single-box-certificate} on
$0<\|\p\|\le\rho^\sharp$. The $t$-intervals of the cover include $t=0$, where
the enclosed quantities are the continuous extensions of the difference
quotients \eqref{eq:nu-L-definitions}.}
\label{tab:local-second-order-certificate}
\end{table}

Algorithm~\ref{alg:single-box-certificate} is implemented by the
routine \path{qn_single_box_certificate}, together with \path{veigs}.
On each box,
\eqref{eq:root-inclusion-test}, \eqref{eq:Fhat-comparison-test},
$M\succ0$, $\inf[\lambda_1]>0$, and \eqref{eq:S-positive-target} hold;
\eqref{eq:subdivision-inclusion} justifies every inherited enclosure.
Moreover, the Gershgorin bound gives
$M\succeq 0.0521I_5$, while
the verified enclosure obtained by
\path{veigs} gives $\lambda_3\ge17.834$. Hence $\lambda_3>16$ and
$\lambda_3-3\pi^2/2\ge 3.030$.
The same inequalities separate $\lambda_1,\lambda_2$ from
\(\lambda_3\).

\begin{corollary}[Separation from the third eigenvalue]
\label{cor:local-cluster-gap}
On $\overline{B(\bzero,\rho^\sharp)}$, we have
$M(\p)\succ0$, $K(\p)\succ0$, and
$\lambda_2(\p)<\lambda_3(\p)$.
Consequently, the closed ball has an open neighborhood $U$ on which
$\{\lambda_1,\lambda_2\}$ is isolated and its symmetric functions,
including $f_2$, are analytic.
\end{corollary}

\begin{proof}
\[
\begin{array}{ll}
\p=\bzero:
 &(\lambda_1,\ldots,\lambda_5)
   =(\pi^2,\pi^2,2\pi^2,4\pi^2,4\pi^2),\\[2mm]
0<\|\p\|\le\rho^\sharp:
 &\lambda_2(t\e)<\pi^2+t\nu_+<\lambda_3(t\e)
 \quad\text{by Theorem~\ref{thm:single-box-enclosures}
 and Table~\ref{tab:local-second-order-certificate}.}
\end{array}
\]
Together with $M(\p)\succ0$ and $K(\p)\succ0$, this gives
$\min_{\|\p\|\le\rho^\sharp}
\bigl(\lambda_3(\p)-\lambda_2(\p)\bigr)>0$.
The last assertion follows from analytic perturbation theory for the
symmetric generalized eigenvalue problem with $M(\p)\succ0$.
\end{proof}

\endgroup

\section{Certified computation for the global step
(Theorem~\ref{t:box-cover-terminates})}\label{app:implementation}

\begingroup
\fontsize{10pt}{11pt}\selectfont
\setlength{\abovedisplayskip}{6pt plus 2pt minus 1pt}
\setlength{\belowdisplayskip}{6pt plus 2pt minus 1pt}
\setlength{\abovedisplayshortskip}{3pt plus 1pt}
\setlength{\belowdisplayshortskip}{4pt plus 1pt minus 1pt}
\setlength{\jot}{2pt}
\setlength{\intextsep}{8pt plus 2pt minus 2pt}

This appendix documents the per-box test of Section~\ref{s:GlobalMax},
as implemented in the supplementary code~\cite{SupplementaryCode} in
MATLAB with \textsc{Intlab}~\cite{Rump1999INTLAB} and the verified
generalized-eigenvalue solver
\texttt{veigs}~\cite{LiuYanagisawa2025veigs}.

Write
\begin{equation}\label{eq:global-box-notation}
 \mathcal B
 :=\{\p_c+\delta:|\delta_r|\le h_r,\ r=1,\ldots,4\},
 \qquad \p_c=(p_{c,1},\ldots,p_{c,4}),
 \qquad [p_r]_{\mathcal B}:=[p_{c,r}-h_r,p_{c,r}+h_r].
\end{equation}
For a scalar-, vector-, or matrix-valued function $F$, the notation
$[F]_{\mathcal B}$ denotes an interval enclosure satisfying
$\{F(\p):\p\in\mathcal B\}\subseteq[F]_{\mathcal B}$, interpreted
entrywise for vectors and matrices.

\subsection*{Enclosure of the pencil}

Using $\Lambda_{m,n}$ and $J$ from Section~\ref{app:pullback}, with the
modes enumerated by $\mathcal I_2$ as in Section~\ref{app:integral-rep}
so that $\Lambda_1,\ldots,\Lambda_5$ are single-indexed, set
\[
 \begin{aligned}
 R_{ij}(\p)&:=\int_\square \Lambda_i\Lambda_jJ\,du\,dv,
 &m_i(\p)&:=\int_\square \Lambda_iJ\,du\,dv,
 \qquad i,j=1,\ldots,5,\\
 q(\p)&:=|Q_\p|=1-a^2-d^2,
 &M(\p)&=R(\p)-\frac{m(\p)m(\p)^\top}{q(\p)}.
 \end{aligned}
\]
For $F\in\{K,R,m\}$, interval forward differentiation gives the centered
enclosure
$\{F(\p):\p\in\mathcal B\}
\subseteq
[F]_{\mathcal B}
:=[F(\p_c)]+\sum_{r=1}^4[\partial_rF]_{\mathcal B}[-h_r,h_r]$.
The entries of $[F(\p_c)]$ and $[\partial_rF]_{\mathcal B}$ are computed by a
$20\times20$ tensor Gauss--Legendre rule and include rigorous bounds for the
quadrature remainders.  Interval arithmetic then gives
$[q]_{\mathcal B}:=1-[a]_{\mathcal B}^2-[d]_{\mathcal B}^2$ and
$[M]_{\mathcal B}:=[R]_{\mathcal B}
-[m]_{\mathcal B}[m]_{\mathcal B}^\top/[q]_{\mathcal B}$.
Write
$[q]_{\mathcal B}=[\underline q_{\mathcal B},\overline q_{\mathcal B}]$
(the upper endpoint $\overline q_{\mathcal B}$ is the $Q(\mathcal B)$ of
Proposition~\ref{p:box-bound})
and $[\pi^2]=[\underline{\pi^2},\overline{\pi^2}]$.
The notation $[M]_{\mathcal B}\succ0$ means that every symmetric matrix in
$[M]_{\mathcal B}$ is positive definite; this is certified by interval
$LDL^\top$ factorization.

Suppose $\underline q_{\mathcal B}>0$ and
$[M]_{\mathcal B}\succ0$.  Applied to the interval pencil
$([K]_{\mathcal B},[M]_{\mathcal B})$, \texttt{veigs} returns an interval
$[\lambda]_{\mathcal B}
=[\underline\lambda_{\mathcal B},\overline\lambda_{\mathcal B}]$
and certified index data $J_{\mathcal B}\subseteq\{1,\ldots,5\}$,
rigorous for \emph{every} point pencil contained in the interval matrices;
the only property used in the proof is that $1\in J_{\mathcal B}$ and
$\lambda_1(\p)\in[\lambda]_{\mathcal B}$ for every $\p\in\mathcal B$.
The library \texttt{veigs} implements the verified computation method by
\cite{behnke1991calculation}; see also
\cite{yamamoto2001simple}.
The implementation uses the interval arithmetic library
\textsc{Intlab}~\cite{Rump1999INTLAB}.
Define
$\Delta_{\mathcal B}
:=\underline{\pi^2}
-\overline q_{\mathcal B}\,\overline\lambda_{\mathcal B}$.
The box $\mathcal B$ is accepted precisely when
\begin{equation}\label{eq:global-code-test}
 \underline q_{\mathcal B}>0,
 \qquad [M]_{\mathcal B}\succ0,
 \qquad 1\in J_{\mathcal B},
 \qquad \underline\lambda_{\mathcal B}>0,
 \qquad \Delta_{\mathcal B}>0.
\end{equation}
For every accepted box and every $\p\in\mathcal B$,
\begin{equation}\label{eq:global-accepted-box-conclusion}
 0<q(\p)\lambda_1(\p)
 \le\overline q_{\mathcal B}\,\overline\lambda_{\mathcal B}
 =\underline{\pi^2}-\Delta_{\mathcal B}
 <\underline{\pi^2}\le\pi^2.
\end{equation}

\subsection*{Finite cover and subdivision}

Let $\mathcal G_0$ be the $3^4$ uniform subdivision of $P_{\mathrm C}$ from
\eqref{e:cube}.  For $\mathcal B$ as in
\eqref{eq:global-box-notation}, define
$r_{\mathcal B}:=\bigl(\sum_{r=1}^4(|p_{c,r}|+h_r)^2\bigr)^{1/2}$
and
$d_{ij}(\p):=|\v_i(\p)-\v_j(\p)|^2$, and let
$[c_i]_{\mathcal B}
=[\underline c_{i,\mathcal B},\overline c_{i,\mathcal B}]$ and
$[d_{ij}]_{\mathcal B}
=[\underline d_{ij,\mathcal B},\overline d_{ij,\mathcal B}]$
be certified enclosures of the defining polynomials of
$\P_{\mathrm K}$.  A box is discarded when at least one of the following
conditions holds:
\begin{equation}\label{eq:global-discard-test}
 r_{\mathcal B}\le\frac{\rho^\sharp}{2},
 \qquad
 \min_{1\le i\le4}\overline c_{i,\mathcal B}<0,
 \qquad
 \max_{1\le i<j\le4}\underline d_{ij,\mathcal B}>T.
\end{equation}
The first condition gives
$\mathcal B\subseteq\overline{B(\bzero,\rho^\sharp/2)}$; either of the other
two gives $\mathcal B\cap\P_{\mathrm K}=\varnothing$.  After this test, one
representative of each $D_4$-orbit is retained; the initial active family has
$16$ boxes.

A box satisfying neither \eqref{eq:global-discard-test} nor
\eqref{eq:global-code-test} is bisected.  The selected coordinate $r$ satisfies
$h_r\ge\tfrac12\max_{1\le s\le4}h_s$.
With $e_r$ the $r$th coordinate vector, the two children are
\begin{equation}\label{eq:global-bisection}
 \mathcal B^{\pm}:=\Bigl\{\p_c\pm\tfrac{h_r}{2}e_r+\delta:
 |\delta_r|\le\tfrac{h_r}{2},\ |\delta_s|\le h_s\ (s\ne r)\Bigr\},
 \qquad
 \mathcal B=\mathcal B^-\cup\mathcal B^+.
\end{equation}
Thus every subdivision preserves the covered set. The choice of bisection
coordinate and the initial orbit reduction are floating-point heuristics:
only the final list of boxes enters the proof, each retained box being
tested by the certified conditions \eqref{eq:global-code-test}.  Let
$\mathcal A$, $\mathcal D$, and $\mathcal U$ denote the accepted boxes,
the discarded boxes, and the boxes satisfying neither test at the maximum
allowed depth $60$; let $N_{\mathrm{bisect}}$ be the number of bisections
and $d_{\max}$ the maximum attained depth.  The finite certificate is
\begin{equation}\label{eq:global-completion-test}
 \mathcal U=\varnothing,
 \qquad
 \Delta_*:=\min_{\mathcal B\in\mathcal A}\Delta_{\mathcal B}>0.
\end{equation}
The certified output is
\[
 |\mathcal A|=114{,}627,
 \quad |\mathcal D|=15{,}623,
 \quad |\mathcal U|=0,
 \quad N_{\mathrm{bisect}}=130{,}234,
 \quad d_{\max}=30,
 \quad \Delta_*\ge2.654\times10^{-5}.
\]

Since $\P_{\mathrm K}\subset P_{\mathrm C}$ and each bisection satisfies
\eqref{eq:global-bisection}, equations
\eqref{eq:global-discard-test} and \eqref{eq:global-completion-test} give
\begin{equation}\label{eq:global-final-cover}
 \Omega_{\mathrm{II}}
 \subseteq
 \overline{B(\bzero,\rho^\sharp/2)}
 \cup
 \bigcup_{g\in D_4}\ \bigcup_{\mathcal B\in\mathcal A}g\mathcal B.
\end{equation}
On the first set in \eqref{eq:global-final-cover}, Step~(I) gives
$|Q_\p|\mu_1(Q_\p)<\pi^2$ for every nonzero $\p$.  On each accepted box $\mathcal B$,
\eqref{eq:global-accepted-box-conclusion} holds for every $\p\in\mathcal B$, and
for $\p\in\mathcal B\cap\Omega_{\mathrm{II}}\subset\P$ the Rayleigh--Ritz bound
applies, giving
$|Q_\p|\mu_1(Q_\p)
\le |Q_\p|\lambda_1(\p)<\pi^2$ there.
Both sides are invariant under $D_4$.  Hence
$|Q_\p|\mu_1(Q_\p)<\pi^2$ for all
$\p\in\Omega_{\mathrm{II}}$, as asserted in
Theorem~\ref{t:box-cover-terminates}.

\endgroup

\end{document}